\documentclass[12pt]{amsart}

\usepackage{amsmath,bm}
\usepackage[color=yellow,textwidth=3.5cm,colorinlistoftodos]{todonotes}
\usepackage{hyperref}
\usepackage[nameinlink,capitalize]{cleveref}
\usepackage{dsfont}
\usepackage{xfrac}
\usepackage[english]{babel}

\newcommand{\summe}[2]{\displaystyle\sum_{#1}^{#2}}
\newcommand{\summezwei}[2]{\sum_{#1}^{#2}}

\newcommand{\integral}[4]{\displaystyle\int\limits_{#1}^{#2}{#3}\,{#4}}

\newcommand{\indikator}[2]{\mathds{1}_{{#1}}({#2})}
\newcommand{\indikatorzwei}[1]{\mathds{1}_{{#1}}}

\newcommand{\erwartung}[1]{\mathbb{E}({#1})}

\newcommand{\limes}[2]{\lim \limits_{#1 \rightarrow #2}}

\newcommand{\limesinf}[2]{\liminf \limits_{#1 \rightarrow #2} \,}

\newcommand{\ft}[1]{\widehat{{#1}}}
\newcommand{\B}{\mathcal{B}}

\newcommand{\C}{\mathbb{C}}

\newcommand{\R}{\mathbb{R}}

\newcommand{\N}{\mathbb{N}}

\newcommand{\Pro}{\mathbb{P} } 
\newcommand{\eps}{\varepsilon} 
\newcommand{\secret}[1]{}

\newcommand{\norm}[1]{\| {#1}\|}
\newcommand{\Li}[1]{\text{L}(\R^{#1})}

\newcommand{\skp}[2]{ \langle {#1},{#2} \rangle}

\newtheorem{theorem}{Theorem}[section]
\newtheorem{theorem2}[theorem]{Theorem and Definition}
\newtheorem{lemma}[theorem]{Lemma}
\newtheorem{prop}[theorem]{Proposition}
\newtheorem{cor}[theorem]{Corollary}

\theoremstyle{definition}
\newtheorem{defi}[theorem]{Definition}
\newtheorem{example}[theorem]{Example}

\theoremstyle{remark}
\newtheorem{remark}[theorem]{Remark}

\numberwithin{equation}{section}
\begin{document}
\sloppy
\title[]{Multi operator-stable random measures and fields}
\allowdisplaybreaks
\author{D. Kremer}
\address{Dustin Kremer, Department Mathematik, Universit\"at Siegen, 57068 Siegen, Germany}
\email{kremer\@@{}mathematik.uni-siegen.de}
\allowdisplaybreaks
\author{H.-P. Scheffler}
\address{Hans-Peter Scheffler, Department Mathematik, Universit\"at Siegen, 57068 Siegen, Germany}
\email{scheffler\@@{}mathematik.uni-siegen.de}

\date{\today}


%Abstract
\begin{abstract}
In this paper we construct vector-valued multi operator-stable random measures that behave locally like operator-stable random measures. The space of integrable functions is characterized in terms of a certain quasi-norm. Moreover, a multi operator-stable moving-average representation of a random field is presented which behaves locally like an operator-stable random field which is also operator-self-similar.
\end{abstract}
\maketitle
%% Section 1: Introduction
\section{Introduction}
The notion of self-similarity of stochastic processes and random fields has a long history and yields to a rich class of stochastic models with various applications. See \cite{Abry}, \cite{Embrechts} and \cite{Vehel} for instance. Recall from \cite{lixiao} that an $\R^m$-valued random field $\mathbb{X}=\{X(t):t \in \R^d\}$ is called \textit{$(E,D)$-operator-self-similar}, if there exist a $d \times d$-matrix $E$ and a $m \times m$-matrix $D$ such that 
\begin{equation} \label{eq:25091801}
\{X(r^E t): t \in \R^d\} \overset{fdd}{=} \{r^D X(t):t \in \R^d\}
\end{equation}
for all $r>0$ holds true. In \cite{lixiao} a so called \textit{moving-average} and a \textit{harmonizable} representation based on \textit{symmetric $\alpha$-stable} ($S \alpha S$) random measures are presented. Multivariate $S \alpha S$ distributions are a special case of the much larger class of \textit{operator-stable} laws. For a comprehensive introduction see \cite{thebook}. Based on the recently developed theory of multivariate \textit{independently scattered random measures} (short: ISRMs) and their integrals (see \cite{paper}), moving-average and harmonizable representations of $(E,D)$-operator-self-similar random fields with operator-stable marginals are presented in \cite{paper2}.\\
Our main feature of $S \alpha S$ ISRMs or more generally of operator-stable ISRMs with exponent $B$ (a $m \times m$-matrix) is that they are homogeneous in the time variable $t \in \R^d$, that is $\alpha$ or $B$ are constant and do not depend on $t$. Motivated by various applications, in \cite{falconer2} scalar-valued so called \textit{multi-stable} random measures were constructed. There the stability index $\alpha: \R \rightarrow \R_{+}$ can vary with time $t$. Based on random integrals of deterministic functions so called multi-stable processes were introduced. \\
An important feature of multi-stable random measures and multi-stable processes is that they behave locally like $\alpha(t)$-stable random measures and $\alpha(t)$-stable processes close to time $t$. That means that the local scaling limits are $\alpha(t)$-stable, but where the stability index varies with $t$. See Theorem 2.7 and Theorem 3.2 in \cite{falconer2}.\\
The purpose of this paper is to generalize some of the results in \cite{falconer2} to the operator-stable setting, by allowing the operator-stable exponent $B(t)$ to depend on the time variable $t$.\\
This paper is organized as follows: In \cref{chapter2}, by applying the general theory of multivariate ISRMs in \cite{paper}, multi operator-stable random measures and integrals are constructed. The set of integrable functions is characterized in terms of a quasi-norm. Moreover, it is shown in \cref{09011805}, that under a mild condition multi operator-stable ISRMs behave locally like an operator-stable ISRM. In \cref{chapter3} a multi operator-stable moving-average representation of a random field $\mathbb{X}=\{X(t):t \in \R^d\}$ is presented and its basic properties are analyzed. Finally, in \cref{chapter4}, we show that $\mathbb{X}$ behaves locally like an operator-stable random field which is $(E,D)$-operator-self-similar. 
%% Section 2: 
\section{Multi operator-stable random measures and integrals} \label{chapter2}
Using the general theory of ISRMs developed in \cite{paper}, we now construct a large class of multivariate independently scattered random measures which are first of all infinitely-divisible. At the same time we refer the reader to \cite{JuMa93}, \cite{paper2} and \cite{thebook} concerning more details about operator-stable distributions which will play a crucial role throughout the paper. \\
In order to start with some notation, let $\Li{m}$ be the set of all linear operators on $\R^m$ and $\text{G}\Li{m}$ the corresponding subset of all invertible operators, represented as $m \times m$-matrices in each case. We write $\norm{\cdot}$ for the Euclidean norm on $\R^m$ (with inner product $\skp{\cdot}{\cdot}$) as well as for the operator norm on $\Li{m}$. Then, given some $D \in \Li{m}$, we denote its adjoint by $D^{*}$ and its trace by $\text{tr}(D)$, respectively. Moreover, let $\text{spec}(D)$ be the set of all eigenvalues of $D \in \Li{m}$. Hence we can define $ \lambda_D:= \min \{\text{Re } \lambda: \lambda \in \text{spec}(D)\} $ and $ \Lambda_D:= \max \{\text{Re } \lambda: \lambda \in \text{spec}(D)\}$. As usual the matrix exponential is given by 
\begin{equation*}
r^D=\exp((\ln r) D)=\summe{k=0}{\infty} \frac{ (\ln r)^k D^k}{k!} \quad \in \text{GL}(\R^m)
\end{equation*}
for any $r>0$ (see Proposition 2.2.2 in \cite{thebook} for more details). In this context we should recall from \cite{bms} that, given an operator $D \in Q(\R^m):=\{E \in \Li{m}: \lambda_D>0\}$, there exists a norm $\norm{\cdot}_D$ on $\R^m$ with unit sphere $S_D:=\{x \in \R^m: \norm{x}_D=1 \}$ such that the mapping $\Psi:(0,\infty) \times S_D \rightarrow \Gamma_m:=\R^m \setminus \{0\}$, defined by $(r,\theta) \mapsto r^{D} \theta$,  becomes a homeomorphism. We call $(\tau_D(x),l_D(x)):=\Psi^{-1}(x)$ the \textit{generalized polar coordinates of $x$ with respect to $D$}. Obviously we have $S_D=\{x:\tau_D(x)=1\}$ as well as $\tau_D(-x)=\tau_D(x)$ and $\tau_D(r^Ex)=r \, \tau_D(x)$ for any $r>0,x \in \Gamma_m$ by uniqueness, respectively. Hence $\tau_D$ is a so called \textit{$D$-homogeneous function} which even becomes continuous on $\R^m$, if we let $\tau_D(0)=0$. \\ 
If not stated otherwise, let $(S,\Sigma, \nu)$ be a $\sigma$-finite measure space and for any $s \in S$ we consider a symmetric operator $B(s) \in Q(\R^m)$ such that $s \mapsto B(s)$ is measurable. Then we define $\lambda(s):=\lambda_{B(s)}$ as well as $\Lambda(s):=\Lambda_{B(s)}$ and assume that 
\begin{equation*}
1/2 < \inf_{s \in S} \lambda(s) \le \sup_{s \in S} \Lambda(s)< \infty. 
\end{equation*}
Hence with $a:=\inf_{s \in S} \Lambda(s)^{-1}$ and $ b:=\sup_{s \in S} \lambda(s)^{-1}$ we have for any $s \in S$ that
\begin{equation} \label{eq:12071805}
  b^{-1} \le \lambda(s) \le \Lambda(s) \le a^{-1}, \quad \text{where} \quad 0<a \le b <2 .
\end{equation}
\begin{remark} \label{17111701}
Fix $s \in S$ and assume that $B(s)=\text{diag}(b_1(s),...,b_m(s))$ is diagonal. Then we get that $r^{B(s)}=\text{diag}(r^{b_1(s)},...,r^{b_m(s)})$ and the mapping $(0,\infty) \ni r \mapsto \norm{r^{B(s)} x}$ is strictly increasing for every $x \ne 0$. On the other hand, if $B(s)$ is merely symmetric, there exists an orthogonal matrix $O(s) \in$ GL$(\R^m)$ such that $r^{B(s)} = O(s) r^{D(s)} O(s)^{*}$ for some diagonal matrix $D(s)$ with $\text{spec}(D(s))=\text{spec}(B(s))$, see Proposition 2.2.2 in \cite{thebook}. Hence the previous observation remains true, since 
\begin{equation*}
\norm{O(s) r^{D(s)} O(s)^{*}x}=\norm{r^{D(s)} O(s)^{*}x} = \norm{r^{D(s)} y(s)}
\end{equation*}
for $y(s)= O(s)^{*}x \ne 0$. Overall and according to Lemma 6.1.5 in \cite{thebook} this allows us to assume that $\norm{\cdot}_{B(s)}=\norm{\cdot}$ and that $S_{B(s)}=S^{m-1}=\{x \in \R^m:\norm{x}=1\}$ for any $s \in S$. 
\end{remark}
Moreover, let $\sigma$ be a finite and symmetric measure on $(S^{m-1}, \B(S^{m-1}))$, where $\B(S^{m-1})$ denotes the corresponding collection of Borel sets. Then for every $s \in S$ the mapping
\begin{equation} \label{eq:15111730}
\varphi(s,C):=\integral{0}{\infty}{   \integral{S^{m-1}}{}{\indikator{C}{r^{B(s)} \theta} r^{-2}}{\sigma(d \theta)} }{dr}, \quad C \in \B(\R^m)
\end{equation}
defines a symmetric \textit{L\'{e}vy measure} on $\R^m$ as the proof of \cref{04101702} below will reveal. This means that $\varphi(s,\{0\})=0$ and $\int_{\R^m} \min \{1,\norm{x}^2\} \, \varphi(s,dx)< \infty$. By a change of variables we also observe that $t \cdot \varphi(s,C)= \varphi(s,t^{-B(s)}(C))$ for every $t>0$ and $C \in \B(\R^m)$. Hence Proposition 4.3.2 in \cite{JuMa93} and Theorem 3.1.11 in \cite{thebook} imply that for every $s \in S$ 
\begin{equation} \label{eq:17111720}
\psi_s(u):=  \integral{\R^m}{}{(\cos \skp{x}{u} -1 ) }{\varphi(s,dx)}, \quad u \in \R^m
\end{equation}
is the \textit{log-characteristic function} of a symmetric operator-stable distribution $\mu_s$ on $\R^m$, i.e. its Fourier transform is given by $\exp(\psi_s)$ and $\mu_s$ is particularly infinitely-divisible. In view of \cref{17111701} and Theorem 7.2.5 in \cite{thebook} we refer to $\sigma$ as the \textit{spectral measure} of $\mu_s$. Also note that $B(s)$ is called an \textit{exponent} of $\mu_s$ and that $\psi_s$ is a $B(s)$-homogeneous function. Finally $\mu_s$ should be \textit{full} in the sense of \cite{thebook} and therefore not concentrated on any hyperplane in $\R^m$. For this purpose we assume that the linear span of the support of $\sigma$ equals $\R^m$. \\ \quad \\
Turning over to independently scattered random measures recall the notation and results of section 3 in \cite{paper}. Consider the set $\mathcal{S}:=\{A \in \Sigma: \nu(A)< \infty\}$ and observe that this is a so called \textit{$\delta$-ring} on $S$ whose generated $\sigma$-algebra equals $\Sigma$ (see \cite{paper} again). Moreover, for every $A \in \mathcal{S}$ we can define a measure $\phi_A$ on $\R^m$ via $\phi_A(C):=\int_A \varphi(s,C) \, \nu(ds)$, since $s \mapsto B(s)$ is measurable. This completes our list of assumptions and we obtain the following result.
\begin{theorem2} \label{04101702}
Under the previous assumptions there exists an $\R^m$-valued ISRM on $\mathcal{S}$ which we denote by $\mathbb{M}=\{M(A):A \in \mathcal{S}\}$ and whose marginal distributions are uniquely determined by the relation $M(A) \sim [0,0,\phi_A]$ for every $A \in \mathcal{S}$. According to Theorem 3.1.11 in \cite{thebook} this means that the characteristic function of $M(A)$ is given by
\begin{equation} \label{eq:15111701}
\R^m \ni u \mapsto  \exp \left( \, \, \integral{\R^m}{}{(\cos \skp{x}{u} -1 ) }{\phi_A(dx)} \right)=\exp \left( \, \, \integral{A}{}{\psi_s(u) }{\nu(ds)} \right).
\end{equation}
We call $\mathbb{M}$ the multi operator-stable ISRM that is generated by $(\mu_s)_{s \in S}$ and $\nu$, abbreviated by $\mathbb{M} \sim (\mu_s,\nu)$. If $(\mu_s)_{s \in S}$ is constant with exponent $B$ we write $\mathbb{M} \sim (B,\sigma,\nu)$ instead.
\end{theorem2}
\begin{proof}
Let $c(s):=\int_{\R^m} \min \{1,\norm{x}^2\} \, \varphi(s,dx)$ and fix $A \in \mathcal{S}$. Then we will first prove that $\phi_A$ is a L\'{e}vy measure. By definition of $\phi_A$ and in view of $\nu(A)<\infty$ this reduces to the finiteness of $\sup_{s \in S} \, c(s)$, since
\begin{equation*} 
\integral{\R^m}{}{\min \{1,\norm{x}^2\}}{\phi_A(dx)}=\integral{A}{}{\integral{\R^m}{}{\min \{1,\norm{x}^2\}}{\varphi(s,dx)}  }{\nu(ds)} \le \nu(A) \cdot \sup_{s \in S} \,c(s).
\end{equation*}
Therefore fix $s \in S$. As in \cref{17111701} we may assume that $B(s)$ is even diagonal, say $B(s)=\text{diag}(b_1(s),...,b_m(s))$. Then for every $r_0>0$ we get that
\begin{equation} \label{eq:15111704}
 \norm{r^{B(s)}} \le C_0 \norm{r^{B(s)}}_1 \le C_0 \max \{r_0^{b_1(s)},...,r_0^{b_m(s)}\} \sum_{j}^{m} (r/r_0)^{b_j(s)} \le \begin{cases} C_1 r^{\lambda(s)}, & \text{$r \le r_0$} \\ C_2 r^{\Lambda(s)}, & \text{$r > r_0$} \end{cases}
 \end{equation}
by equivalence of norms. Here $C_0,C_1,C_2>0$ are suitable constants which only depend on $r_0$, because $b^{-1} \le b_j(s) \le a^{-1}$ for $j=1,...,m$. Recall \eqref{eq:12071805} and \eqref{eq:15111730}. Then, using \eqref{eq:15111704} for $r_0=1$, the following computation holds true, where we may assume that $C_1,C_2 \ge 1$.
\allowdisplaybreaks
\begin{align}
c(s)&= \integral{0}{\infty}{ \integral{S^{m-1}}{}{\min \{1,\norm{r^{B(s)}\theta}^2 \}r^{-2} }{\sigma(d \theta)} }{dr} \notag \\
& \le C_1^2 \integral{0}{1}{ \integral{S^{m-1}}{}{\min \{1,r^{2 \lambda(s)} \}r^{-2} }{\sigma(d \theta)} }{dr} + C_2^2 \integral{1}{\infty}{ \integral{S^{m-1}}{}{\min \{1,r^{2 \Lambda(s)} \}r^{-2} }{\sigma(d \theta)} }{dr} \notag \\
& \le C_1^2 \sigma(S^{m-1}) \integral{0}{1}{ r^{2/b-2} }{dr} + C_2^2 \sigma(S^{m-1})  \integral{1}{\infty}{ r^{-2} }{dr}  \label{eq:15111710}.
\end{align}
Note that \eqref{eq:15111710} is independent of $s$ and at the same time finite, since $b<2$. This shows that $\sup_{s \in S} \, c(s)<\infty$. Overall Theorem 3.1 in \cite{paper} implies the existence of some suitable probability space $(\Omega,\mathcal{A},\Pro)$ with $\mathbb{M}$ as asserted.
\end{proof}
\begin{remark} \label{16111701}
The fullness of $\varphi(s,\cdot)$ implies that $c(s)>0$ for every $s \in S$. Therefore $\nu$ is equivalent to the measure $\lambda_{\mathbb{M}}(ds):=c(s) \nu(ds)$. Under the given assumptions it can be easily verified that $\lambda_{\mathbb{M}}$ equals the so called \textit{control measure} of $\mathbb{M}$ in the sense of Theorem 3.2 in \cite{paper}. Also recall the related mappings $\rho_{\mathbb{M}}:S \times \B(\R^m) \rightarrow [0,\infty]$ and $K_{\mathbb{M}}:S \times \R^m \rightarrow \C$ from Proposition 3.5 in \cite{paper}. Actually, in the present case they can be computed as 
\begin{equation*}
\rho_{\mathbb{M}}(s,C)=c(s)^{-1}\varphi(s,C) \quad \text{and} \quad K_{\mathbb{M}}(s,u)=c(s)^{-1} \psi_s(u),
\end{equation*}
since the distributions $\mu_s$ and $M(A)$ are symmetric for every $s \in S$ and $A \in \mathcal{S}$, respectively. 
\end{remark} 
The two subsequent results use parts of the previous proof and will be essential afterwards. For this purpose we let $\norm{\cdot}^{a,b}:=\norm{\cdot}^a+\norm{\cdot}^b$ as well as $(\tau_s (\cdot),l_s(\cdot)):=(\tau_{B(s)}(\cdot),l_{B(s)}(\cdot))$.
\begin{cor} \label{24111750}
Fix $r_0>0$. Then, with the corresponding constants from \eqref{eq:15111704}, we have for all $\norm{x} \le r_0$ and $s \in S$ that
\begin{equation} \label{eq:24111730}
C_1(r_0)^{- 1/ \lambda(s)} \norm{x}^{1/ \lambda(s)} \le \tau_s(x) \le C_2(r_0^{-1})^{1/ \Lambda(s)}\,  \norm{x}^{1/ \Lambda(s)} 
\end{equation}
and, for all $\norm{x} \ge r_0$ and $s \in S$, that
\begin{equation} \label{eq:24111731}
C_1(r_0)^{- 1/ \lambda(s)} \norm{x}^{1/ \Lambda(s)} \le \tau_s(x) \le C_1(r_0^{-1})^{1/ \Lambda(s)}\,  \norm{x}^{1/ \lambda(s)} .
\end{equation}
Particularly, there exist constants $C_3,C_4>0$ such that
\begin{equation} \label{eq:24111732}
C_3 \min \{\norm{x}^{a}, \norm{x}^{b} \}  \le \tau_s(x) \le C_4 \max \{\norm{x}^{a}, \norm{x}^{b} \} \le C_4  \norm{x}^{a,b}
\end{equation}
holds for any $s \in S$ and $x \in \R^m$.
\end{cor}
\begin{proof}
In view of \cref{17111701} and \eqref{eq:15111704} the inequalities in \eqref{eq:24111730} and \eqref{eq:24111731} follow by a straight-forward extension of the proof of Lemma 2.1 in \cite{bms}. Due to our assumptions on $B(s)$ they also yield the additional statement (consider $r_0=1$ for example).
\end{proof}
\begin{lemma} \label{23111710}
For any $0<\rho_1 \le \rho_2$ we define the set $T_{\rho_1,\rho_2}:= \{x \in \R^m: \rho_1 \le \norm{x} \le \rho_2\}$. Then there exist $K_1,K_2>0$ (depending on $\rho_1,\rho_2$ only) such that 
\begin{equation*}
K_1 \le |\psi_s(u)| \le K_2 \quad \text{for all $s \in S$ and $u \in T_{\rho_1,\rho_2}$ }.
\end{equation*}
\end{lemma}
\begin{proof}
The upper bound follows from $\sup_{s \in S} c(s)< \infty$ (see the previous proof) and the Cauchy-Schwarz inequality, since we have
\begin{equation} \label{eq:13071801}
|\psi_s(u)| = \integral{\R^m}{}{(1-\cos \skp{x}{u} ) }{\varphi(s,dx)} \le 2 \integral{\R^m}{}{ \min \{1,\skp{x}{u}^2 \}}{\varphi(s,dx)} \le 2 (1+\rho_2^2) \, c(s)
\end{equation}
for any $u \in T_{\rho_1,\rho_2}$. For the lower bound we first note that $\psi_s(u) \ne 0 $ for any $s \in S$ and $u \ne 0$ due to Corollary 7.1.12 in \cite{thebook}. Hence, for any $s \in S$, there exists some $u_s \in  T_{\rho_1,\rho_2}$ such that 
\begin{equation*}
m(s):=\min_{u \in T_{\rho_1,\rho_2}} |\psi_s(u)|= - \psi_s(u_s) > 0
\end{equation*}
by continuity. If we suppose that $K_1:= \inf_{s \in S} m(s) =0$, this yields a sequence $(s_n) \subset S$ with $m(s_n) \rightarrow 0$. On the other hand the set 
\begin{equation*}
 \mathcal{C}:= \{\text{$B \in \text{GL}(\R^m)$ symmetric}: b^{-1} \le \lambda_B \le \Lambda_B \le 	a^{-1} \}
 \end{equation*}
 is bounded, since $\norm{B}=\Lambda_{B}$ in this case. Additionally and in view of the \textit{Courant-Fischer Theorem} (see Problem 4.4 in \cite{teschl} for example) it is closed. Hence, without loss of generality, we can assume that $B(s_n) \rightarrow B^{*}$ for some $B^{*} \in \mathcal{C}$. As before this allows us to consider a full and symmetric operator-stable distribution $\mu^{*}$ on $\R^m$ with exponent $B^{*}$ and spectral measure $\sigma$. Denote its log-characteristic function by $\psi^{*}$ and fix $u \in \R^m$. Then, according to \eqref{eq:13071801}, we observe for every $r>0,\theta \in S^{m-1}$ and $n \in \N$ that 
 \begin{equation} \label{eq:13071806}
  1-\cos \skp{r^{B(s_n)} \theta }{u}  \le 2 (1+\norm{u}^2) \min \{1,\norm{r^{B(s_n)} \theta}^2\},
 \end{equation}
where the left hand-side of \eqref{eq:13071806} tends to $1-\cos \skp{r^{B^{*}} \theta }{u}$ as $n \rightarrow \infty$. In view of
\begin{equation*}
- \psi_{s_n}(u) = \integral{0}{\infty}{  \integral{S^{m-1}}{}{  (1-\cos \skp{r^{B(s_n)} \theta }{u} ) r^{-2}}{\sigma(d \theta)} }{dr}
\end{equation*}
we can argue as in the proof of \cref{04101702} to verify that \eqref{eq:13071806} yields $\psi_{s_n}(u) \rightarrow \psi^{*}(u)$ by dominated convergence. This means that $\mu_{s_n} \rightarrow \mu^{*}$ weakly, since $u \in \R^m$ was arbitrary. By passing to a subsequence we can assume that $u_{s_n} \rightarrow u^{*}$ for some $u^{*} \in T_{\rho_1,\rho_2}$. However, using L\'{e}vy's continuity theorem this implies for the corresponding Fourier transforms that
\begin{equation*}
\ft{\mu^{*}}(u^{*})= \limes{n}{\infty} \ft{\mu_{s_n}}(u_{s_n}) = \limes{n}{\infty} \exp (- m(s_n))= 1
\end{equation*}
which contradicts Corollary 7.1.12 in \cite{thebook}, since $u^{*} \ne 0$.
\end{proof}
Next we want to consider stochastic integrals of the form $I_{\mathbb{M}}(f):=I(f):=\int_S f(s) \, \mathbb{M}(ds)$ for suitable matrix-valued functions $f:S \rightarrow \Li{m}$ which leads to $\R^m$-valued random vectors. The underlying definition is quite natural for simple functions $f$ and results from a stochastic limit for general $f \in \mathcal{I}(\mathbb{M})$, where
\begin{equation*}
\mathcal{I}(\mathbb{M}):=\{f: \text{$f$ is measurable and integrable with respect to } \mathbb{M}\}. 
\end{equation*}
Then $\mathcal{I}(\mathbb{M})$ becomes a real vector space and the mapping $f \mapsto I(f)$ is linear on $\mathcal{I}(\mathbb{M})$. We refer to \cite{paper} for more details and subsequently merely apply these results to our setting. Also note that we will identify random vectors which are equal almost surely and that $\mathbb{M}(A)=I(\mathbb{E}_m \indikatorzwei{A})$, where $\mathbb{E}_m$ denotes the identity matrix on $\R^m$. Finally, for $f:S \rightarrow \R$, we interpret $I(f)$ implicitly as $I(f \, \mathbb{E}_m)$. Defining 
\begin{align}
V_{\mathbb{M}}: \Li{m} \times S \rightarrow \R_{+}, \quad &(E,s) \mapsto \integral{\R^m}{}{\min \{1,\norm{Ex}^2\} }{\varphi(s,dx)} , \label{eq:24111701}
\end{align}
we obtain the following result.
\begin{prop} \label{11101701}
Let $\mathbb{M}$ be a multi operator-stable ISRM as in \cref{04101702}.
\begin{itemize}
\item[(a)] For measurable $f:S \rightarrow \Li{m}$ the following three statements are equivalent:
\begin{itemize}
\item[(i)] $f \in \mathcal{I}(\mathbb{M})$.
\item[(ii)] The integral
\begin{align*}
 \integral{S}{}{ V_{\mathbb{M}}(f(s),s) }{\nu(ds)} =  \integral{\R^m}{}{\min \{1, \norm{x}^2\}  }{\phi_f(dx)},
\end{align*}
is finite, where
\begin{equation*}
\phi_f(A):= \integral{S}{}{\integral{\R^m}{}{\indikatorzwei{f(s)x \in A \setminus \{0\}} }{\varphi(s,dx)}  }{\nu(ds)}, \quad A \in \B(\R^m).
\end{equation*}
In this case we have $I(f) \sim [0,0, \phi_f]$.
\item[(iii)] The mapping
\begin{equation*} 
\R^m \ni u \mapsto \integral{S}{}{\psi_s(f(s)^{*}u)}{\nu(ds)}=\integral{\R^m}{}{(\cos \skp{u}{x}-1)}{\phi_f(dx)}
\end{equation*}
is well-defined (i.e. the integral exists for every $u \in \R^m$) and continuous.
\end{itemize}
\item[(b)] Assume that $f_1,...,f_n \in \mathcal{I}(\mathbb{M})$, then the $\R^{n \cdot m}$-valued random vector $(I(f_1),...,I(f_n))$ is infinitely-divisible and its log-characteristic function is given by
\begin{equation} \label{eq:17111740}
\R^{n \cdot m} \ni (u_1,...,u_n) \mapsto \integral{S}{}{\psi_s \left( \summe{j=1}{n} f_j(s)^{*}u_j \right) }{\nu(ds)}.
\end{equation}
\item[(c)] Let $f \in \mathcal{I}(\mathbb{M})$. Furthermore suppose that $f(s) \in \text{GL}(\R^m)$ for every $s \in A$, where $A \subset S$ is measurable with $\nu(A)>0$. Then $I(f)$ is full.
\end{itemize}
\end{prop}
\begin{proof}
Recall \cref{16111701} and that $\mu_s \sim [0,0,\varphi(s,\cdot)]$ with $\varphi(s,\cdot)$ being symmetric for every $s \in S$. Then (a) and (b) follow immediately from Proposition 3.3 and section 5 of \cite{paper}, respectively. Also recall that $\mu_s$ is full for every $s \in S$. Hence, using \eqref{eq:17111740} and Lemma 1.3.11 in \cite{thebook}, a slight refinement of Proposition 2.6 (a) in \cite{paper2} gives part (c). The details are left to the reader.
\end{proof}
For the case $m=1$ monotonicity arguments come into effect and Theorem 3.3 in \cite{RajRo89} implies that $f \in \mathcal{I}(\mathbb{M})$ if and only if $|f| \in \mathcal{I}(\mathbb{M})$. Actually, this allows to identify $\mathcal{I}(\mathbb{M})$ as a so called \textit{Musielak-Orlicz space} (see \cite{musielak} fore more details) in their framework. Unfortunately, this characterization will fail in the case $m>1$, because $f \in \mathcal{I}(\mathbb{M})$ is not equivalent to $\norm{f}  \in \mathcal{I}(\mathbb{M})$ in general. Moreover, \cref{16111703} and \cref{09031801} below will even show that the condition
\begin{equation*}
f_{i,j} \in \mathcal{I}(\mathbb{M}) \quad \text{for every $i,j=1,...,m$}
\end{equation*}
is sufficient, but not necessary for $f=(f_{i,j})_{i,j=1,...,m}$ to be integrable with respect to $\mathbb{M}$. However, for every measurable $f:S \rightarrow \Li{m}$ and $\lambda>0$ we can define  
\begin{equation} \label{eq:16111710}
H(f,\lambda):= \integral{S}{}{ \underset{\norm{u}_{\infty} \le \lambda^{-1}}{\sup} \tau_{s}(f(s)^{*}u)}{\nu(ds)} \quad \in [0,\infty],
\end{equation}
where $\norm{\cdot}_{\infty}$ is the supremum norm on $\R^m$. In order to benefit from $H(f,\lambda)$ we first consider the next statement which is similar to Lemma 4.4 in \cite{paper}.
\begin{lemma}
Let $X$ be an $\R^m$-valued random vector with characteristic function $\omega(\cdot)$. Then there exists a $C>0$ such that the following inequalities hold for any $\delta>0$.
\begin{equation} \label{eq:12101701}
\Pro (\norm{X} \ge  \delta) \le C \, \delta^m \integral{ \norm{u}_{\infty} \le \delta^{-1} }{}{(1-\omega(u))}{du} \le C \, \delta^m \integral{  \norm{u}_{\infty} \le \delta^{-1}}{}{|1-\omega(u)|}{du}.
\end{equation}
\end{lemma}
\begin{proof}
Define $g(y):=\sin(y)/y$ for $y \ne 0$ and $g(0):=1$. Hence $|g(y)| \le 1$, while we can find some $c \in (0,1)$ such that $|g(y)| \le c $ for any $|y| \ge 1/ \sqrt{m} $. Then, using the univariate idea of (1.2) in \cite{Pre56}, we obtain that
\begin{equation} \label{eq:13071810}
 \integral{ \norm{u}_{\infty} \le \delta^{-1}}{}{(1-\omega(u))}{du} = 2^m \delta^{-m} \integral{\R^m}{}{ \left[ 1-\prod_{j=1}^m g(\delta^{-1} x_j) \right ] }{\mathcal{L}(X)(dx)},
\end{equation}
where $x=(x_1,...,x_m)$ and where $\mathcal{L}(X)$ denotes the distribution of $X$. Finally, a truncation of the right-hand side integral in \eqref{eq:13071810} on $ \{x: \norm{x} \ge \delta \} \subset \{x: \norm{x}_{\infty} \ge \delta/ \sqrt{m}\}$ gives the assertion with $C:= (2^m(1- c))^{-1}$. Note that all expressions in \eqref{eq:12101701} are real.
\end{proof}
Now we can prove the main result of this section which, by the way, generalizes Proposition 2.3 in \cite{falconer3}. Here we call $\norm{\cdot}_{\mathbb{M}}$ a \textit{quasi-norm} on $\mathcal{I}(\mathbb{M})$, if it has the usual properties of a norm, except a possible weakening of the triangular inequality, i.e. there is some $T \ge 1$ such that 
\begin{equation} \label{eq:12181010}
\norm{f_1+f_2}_{\mathbb{M}} \le T (\norm{f_1}_{\mathbb{M}}+\norm{f_2}_{\mathbb{M}})
\end{equation}
holds for all $f_1,f_2 \in \mathcal{I}(\mathbb{M})$. Moreover, we identify elements in $ \mathcal{I}(\mathbb{M})$ that are identical $\nu$-a.e., since $\nu$ is equivalent to the control measure $\lambda_{\mathbb{M}}$ and since $I_{\mathbb{M}}(f)=I_{\mathbb{M}}(g)$ a.s. if and only if $f=g$ holds true $\lambda_{\mathbb{M}}$-almost everywhere (see \cite{paper}).
\begin{theorem} \label{16111705}
Let $\mathbb{M}$ be as before. 
\begin{itemize}
\item[(a)] Recall \eqref{eq:16111710}. Then the following identities hold: 
\begin{align}
\mathcal{I}(\mathbb{M})  &= \{f:S \rightarrow \Li{m} \, | \, \text{$f$ is measurable and $H(f,\lambda)< \infty$ for all $\lambda>0$} \}  \notag \\
&= \{f:S \rightarrow \Li{m} \, | \, \text{$f$ is measurable and $H(f,\lambda)< \infty$ for some $\lambda>0$} \}. \notag
\end{align}
\item[(b)] We get a quasi-norm on $\mathcal{I}(\mathbb{M})$ in virtue of
\begin{equation*}
\norm{f}_{\mathbb{M}}:= \inf \{\lambda>0: H(f,\lambda) \le 1 \}.
\end{equation*} 
Moreover, for any $f \in \mathcal{I}(\mathbb{M})$ and $\lambda>0 $ we have
\begin{equation} \label{eq:10101802}
C_3  \min \left \{ \left (\frac{\norm{f}_{\mathbb{M}}}{\lambda} \right)^a , \left (\frac{\norm{f}_{\mathbb{M}}}{\lambda} \right)^b \right \} \le  H(f, \lambda) \le C_4 \max \left \{ \left (\frac{\norm{f}_{\mathbb{M}}}{\lambda} \right)^a , \left (\frac{\norm{f}_{\mathbb{M}}}{\lambda} \right)^b \right \},
\end{equation}
where the constants $C_3,C_4>0$ are those from \eqref{eq:24111732}.
\item[(c)] There exists some $L_1 \ge 1$ such that 
\begin{equation}  \label{eq:12031801}
\Pro(\norm{I(f)} \ge \delta) \le L_1 H(f,\delta)
\end{equation}
holds for any $f \in \mathcal{I}(\mathbb{M})$ and $\delta>0$. Additionally, if $0< p<a$, there exists a constant $L_2=L_2(p)>0$ such that
\begin{equation*}
\erwartung{\norm{I(f)}^p} \le L_2 \norm{f}_{\mathbb{M}}^{p}.
\end{equation*}
\item[(d)] The vector space $\mathcal{I}(\mathbb{M})$ is complete with respect to $\norm{\cdot}_{\mathbb{M}}$ and, for $f,f_1,... \in \mathcal{I}(\mathbb{M})$, we have the characterization
\begin{equation} \label{eq:10101801}
I(f_n) \rightarrow I(f) \text{ in probability } \quad \Leftrightarrow \quad  \norm{f_n-f}_{\mathbb{M}} \rightarrow 0.
\end{equation}
\end{itemize} 
\end{theorem}
\begin{proof}
Recall that $\psi_s$ is $B(s)$-homogeneous with $\psi_s(0)=0$. Therefore we can use \cref{23111710} for $\rho_1=\rho_2=1$ (with $K_1,K_2>0$ accordingly) to see that 
\begin{equation} \label{eq:23111750}
\forall s \in S \, \, \forall u \in \R^m: \qquad K_1 \tau_s(u) \le |\psi_s(u)| \le K_2 \tau_s(u).
\end{equation}
We start with the first identity of part(a). Thus if $H(f,\lambda)< \infty$ for every $\lambda>0$, we observe that (iii) of \cref{11101701} is fulfilled by continuity of $\psi_s$, using dominated convergence. Conversely, assume that  $f \in \mathcal{I}(\mathbb{M})$ and therefore that (ii) of \cref{11101701} holds. Fix $\lambda>0$ and recall \eqref{eq:24111701}. Then, similar to the proof of \cref{23111710} and using \eqref{eq:23111750}, we get the following estimates for all $s \in S$ and $u \in \R^m$ with $\norm{u}_{\infty} \le \lambda^{-1}$.
\begin{align*}
\tau_s(f(s)^{*}u) & \le K_1^{-1} |\psi_s(f(s)^{*}u)| \\
&= K_1^{-1} \integral{\R^m}{}{(1- \cos\skp{f(s)x}{u})}{\varphi(s,dx)} \\
& \le 2 \, K_1^{-1} (1+\sqrt{m} \lambda^{-1}) \, V_{\mathbb{M}}(f(s),s).
\end{align*}
Note that the last expression is integrable with respect to $\nu$ due to our present assumption. Hence we have that $H(f,\lambda)<\infty$ which gives the first identity stated in part (a). Now, without loss of generality, we may assume that $H(f,1)< \infty$ holds true in order to prove the second identity. Fix $\lambda>0$ again. Then, using \eqref{eq:24111732} and the $B(s)$-homogeneity of $\tau_s$, we derive for any $s \in S,u \in \R^m$ and $\gamma>0$ that
\begin{equation} \label{eq:27111701}
C_3 \min \{\gamma^{a}, \gamma^{b} \}  \tau_s(u) \le   \tau_s(\gamma u) = \tau_s(u) \tau_s(\gamma l_s(u))  \le  C_4 \max \{\gamma^{a}, \gamma^{b} \}  \tau_s(u).
 \end{equation}
However, this gives the assertion of part (a), since
\begin{equation*}
H(f,\lambda) = \integral{S}{}{\underset{\norm{u}_{\infty} \le 1 }{\sup} \tau_{s}(\lambda^{-1} f(s)^{*}u)}{\nu(ds)} \le C_4 \max \{\lambda^{-a}, \lambda^{-b} \}  \, H(f,1)< \infty.
\end{equation*}
Concerning $\norm{\cdot}_{\mathbb{M}}$ we first of all see that $\tau_s(0)=0$ implies $\norm{0}_{\mathbb{M}}=0$. Conversely, consider $ f \in \mathcal{I}(\mathbb{M})$ with $\norm{f}_{\mathbb{M}}=0$ and assume that $f \ne 0$ (on a set $A \in \Sigma$ with $\nu(A)>0$). Then it follows that $H(f,\lambda) \le 1$ for all $\lambda>0$ and, as $\lambda \rightarrow 0$, the monotone convergence theorem implies that
\begin{equation} \label{eq:24111755}
H(f):= \integral{S}{}{ \underset{u \in \R^m}{\sup} \, \tau_{s}(f(s)^{*}u)}{\nu(ds)} \le 1.
\end{equation}
On the other hand we have $\sup_{u \in \R^m} \norm{f(s)^{*}u}=\infty$ for every $s \in A$ and therefore $H(f)=\infty$ due to \eqref{eq:24111732}. This contradicts \eqref{eq:24111755} and we obtain that $f=0$. The homogeneity property $\norm{\gamma f}_{\mathbb{M}}=|\gamma| \, \norm{f}_{\mathbb{M}}$ follows from the definition of the infimum, since 
\begin{equation} \label{eq:12101820}
\forall \lambda>0 \, \, \forall \gamma \ne 0: \quad H(\gamma f, \lambda)=H(f, |\gamma|^{-1} \lambda).
\end{equation}
Moreover, using Proposition 1.3.4 in \cite{thebook} and \eqref{eq:17111720}, we see that 
\begin{equation} \label{eq:16081801}
\forall s \in S \, \forall u_1,u_2 \in \R^m: \quad |\psi_s(u_1+u_2)| \le 2 (|\psi_s(u_1)|+|\psi_s(u_2)|).
\end{equation}
In view of \eqref{eq:23111750} we conclude that there exists a constant $C_0>0$ such that 
\begin{equation} \label{eq:27111705}
\forall \lambda>0 \,  \forall f_1,f_2 \in \mathcal{I}(\mathbb{M}): \quad H(f_1+f_2,\lambda) \le C_0 (H(f_1,\lambda)+H(f_2,\lambda)).
\end{equation}
At the same time, for $T \ge 1$ chosen sufficiently large and fixed in sequel, we can use \eqref{eq:27111701} to verify that
\begin{equation*}
\forall \in S \, \forall u \in \R^m: \quad \tau_s(T^{-1}u) \le (2C_0)^{-1} \tau_s(u).
\end{equation*}
In view of \eqref{eq:27111705} this shows for any $\lambda_1,\lambda_2>0$ and $f_1,f_2 \in \mathcal{I}(\mathbb{M})$ that 
\allowdisplaybreaks
\begin{align*}
H(f_1+f_2, T(\lambda_1+\lambda_2)) & \le C_0(H(f_1, T(\lambda_1+\lambda_2))+H(f_2, T(\lambda_1+\lambda_2)))  \\
& \le C_0(H(f_1, T \lambda_1)+H(f_2, T\lambda_2))  \\ 
& = C_0(H(T^{-1} f_1,  \lambda_1)+H(T^{-1} f_2, \lambda_2))  \\
& \le 2^{-1} (H( f_1,  \lambda_1)+ H(f_2, \lambda_2)).
\end{align*}
Consider $\lambda_i=\norm{f_i}_{\mathbb{M}}+\eps$ ($i=1,2)$ to derive that
\begin{equation*}
H(f_1+f_2, T(\norm{f_1}_{\mathbb{M}} + \norm{f_2}_{\mathbb{M}} +2\eps)) \le 1
\end{equation*}
holds for any $\eps>0$. As $\eps \rightarrow 0$ this shows that $\norm{\cdot}_{\mathbb{M}}$ is a quasi-norm on $\mathcal{I}(\mathbb{M})$. Moreover, note that \eqref{eq:10101802} is true for $\norm{f}_{\mathbb{M}}=0$ (which is equivalent to $f=0$ $\nu$-almost everywhere). However, for $\norm{f}_{\mathbb{M}}>0$, \eqref{eq:12101820} implies that $H(f, \lambda)=H(\lambda^{-1} \norm{f}_{\mathbb{M}} f, \norm{f}_{\mathbb{M}} )$ and \eqref{eq:10101802} follows from
 \eqref{eq:27111701}, since $H(f,\norm{f}_{\mathbb{M}})=1$ by continuity of $\lambda \mapsto H(f,\lambda)$. \\
We now prove part (c). Recall that $|1-\exp(z)| \le |z|$ for any $z \in \C$ with $\text{Re }z \le 0$ and that the log-characteristic function of $I(f)$ is given by \eqref{eq:17111740} for $n=1$. Then \eqref{eq:12101701} (with the corresponding constant $C>0$), the triangular inequality and  \eqref{eq:23111750} imply that
\begin{align*}
\Pro(\norm{I(f)} \ge \delta)  & \le C \, \delta^m \integral{  \norm{u}_{\infty} \le \delta^{-1} }{}{\integral{S}{}{ |\psi_s(f(s)^{*}u) | }{\nu(ds)} }{du} \\
& \le C \, K_2 \, \delta^m \integral{  \norm{u}_{\infty} \le \delta^{-1} }{}{\integral{S}{}{ \underset{\norm{y}_{\infty} \le \delta^{-1}}{\sup} \tau_s(f(s)^{*}y) }{\nu(ds)} }{du}  \\
& =  2^m \, C \, K_2 \, H(f,\delta).
\end{align*}
Thus let $L_1=\max \{1,2^m C K_2\}$ to derive \eqref{eq:12031801}. For the second statement of part (c) fix $p>0$ as above. Then the following computation gives the assertion, where we use \eqref{eq:12031801} as well as \eqref{eq:10101802} and where we can assume that $\norm{f}_{\mathbb{M}} >0$.
\allowdisplaybreaks
\begin{align*}
\erwartung{\norm{I(f)}^p} &= p \integral{0}{\infty}{ \lambda^{p-1} \, \Pro(\norm{I(f)} \ge \lambda) }{d \lambda} \\
& \le p L_1 \integral{0}{\norm{f}_{\mathbb{M}} }{\lambda^{p-1}  }{d \lambda} + p L_1 \integral{\norm{f}_{\mathbb{M}} }{\infty}{\lambda^{p-1} H(f,\lambda) }{d \lambda} \\
& \le L_1 \norm{f}_{\mathbb{M}} ^{p}  + p L_1 C_4  \norm{f}_{\mathbb{M}}^a \integral{\norm{f}_{\mathbb{M}} }{\infty}{\lambda^{p-1-a}}{d \lambda} \\
& = L_1 \norm{f}_{\mathbb{M}} ^{p}  + p L_1 C_4 (a-p)^{-1} \norm{f}_{\mathbb{M}} ^{p}   \\
&=: L_2 \norm{f}_{\mathbb{M}}^p.
\end{align*}

In order to show that $\mathcal{I}(\mathbb{M})$ is complete, let $(f_n) \subset \mathcal{I}(\mathbb{M})$ be a Cauchy-sequence with respect to $\norm{\cdot}_{\mathbb{M}}$. Then we have to find some $f \in \mathcal{I}(\mathbb{M})$ such that $\norm{f_n-f}_{\mathbb{M}} \rightarrow 0$ as $n \rightarrow \infty$ or at least along a suitable subsequence, since $(f_n)$ is Cauchy and since $\norm{\cdot}_{\mathbb{M}}$ fulfills \eqref{eq:12181010}. Hence, without loss of generality and throughout following the idea of the proof of Theorem 5.2.1 in \cite{Dud02}, it can be assumed that 
\begin{equation} \label{eq:10101803}
\norm{f_m-f_n}_{\mathbb{M}} \le 2^{-N} \quad \text{for all $n,m \ge N$}.
\end{equation}
Define 
\begin{equation*}
A_n:= \{s \in S:   \underset{\norm{u}_{\infty} \le 1}{\sup} \, \tau_{s}((f_{n+1}(s)-f_{n}(s))^{*}u)) > n^{-2b}  \}, \quad n \in \N.
\end{equation*}
Then, using \eqref{eq:10101802} and \eqref{eq:10101803}, we see for any $n \in \N$ that 
\begin{equation*}
n^{-2b} \nu(A_n) \le H(f_{n+1}-f_{n},1) \le C_4 2^{-an}
\end{equation*}
which implies that $\summezwei{n=1}{\infty} \nu(A_n)<\infty$ and therefore that $B:= \limsup_{n \rightarrow \infty} A_n$ is a $\nu$-null set.  Let $f(s):=0$ for every $s \in B$. However, for $s \in B^c=S \setminus B$, we observe that there exists some $N(s) \in \N$ such that $s \notin A_n$ for all $n \ge N(s)$. In view of \eqref{eq:24111732} and $\norm{\cdot}_{\infty} \le \norm{\cdot}$ it follows for all $s \in B^c, n \ge N(s)$ and $\norm{u} \le 1$ that $\norm{(f_{n+1}(s)-f_n(s))^{*}u}$ is bounded. Using the same argument we obtain a constant $C_1>0$ such that the following inequality holds for all $s \in B^c$ and $n \ge N(s)$:
\begin{equation} \label{eq:12101801}
 \underset{\norm{u}\le 1}{\sup} \norm{(f_{n+1}(s)-f_n(s))^{*}u}^b \le C_1  \underset{\norm{u}_{\infty} \le 1}{\sup} \, \tau_s((f_{n+1}(s)-f_n(s))^{*}u).
\end{equation}
Actually, this implies for all $s \in B^c$ and $n,l \ge N(s)$ that
\allowdisplaybreaks
\begin{align*}
\norm{f_n(s)-f_l(s)} & \le \summe{j=N(s)}{\infty}  \norm{(f_{j+1}(s)-f_{j}(s))^{*}} \\
&= \summe{j=N(s)}{\infty}  \left (\underset{\norm{u}\le 1}{\sup} \norm{(f_{j+1}(s)-f_{j}(s))^{*}u}^b \right)^{1/b} \\
& \le C_{1}^{1/b} \summe{j=N(s)}{\infty} \left( \underset{\norm{u}_{\infty} \le 1}{\sup}  \tau_s((f_{n+1}(s)-f_n(s))^{*}u) \right)^{1/b} \\
& \le C_{1}^{1/b} \summe{j=N(s)}{\infty} j^{-2} 
\end{align*}
due to \eqref{eq:12101801} and since $s \notin A_j$ for all $j \ge N(s)$. In particular we get that $\norm{f_n(s)-f_l(s)}  \rightarrow 0$ as $N(s)$ tends to $\infty$ which shows that $(f_n(s))$ is Cauchy with respect to the operator norm $\norm{\cdot}$. Denote the corresponding limit by $f(s)$ and, overall, observe that $f:S \rightarrow \Li{m}$ is measurable with $f_n(s) \rightarrow f(s)$ $\nu$-almost everywhere. Moreover, \eqref{eq:12181010} and \eqref{eq:10101803} imply for any $n \in \N$ that $\norm{f_n}_{\mathbb{M}} \le T(\norm{f_1}_{\mathbb{M}}+1/2).$ Hence by continuity of $\tau_s$, a routine estimate, Fatou's Lemma and \eqref{eq:10101802} we obtain that
\allowdisplaybreaks 
\begin{align*}
H(f,1) &= \integral{S}{}{\underset{\norm{u}_{\infty} \le 1}{\sup} \, \limesinf{n}{\infty} \tau_s(f(s)^{*}u)}{\nu(ds)} \\
&\le \integral{S}{}{\limesinf{n}{\infty}  \underset{\norm{u}_{\infty} \le 1}{\sup} \tau_s(f(s)^{*}u)}{\nu(ds)} \\ 
&\le  \limesinf{n}{\infty} H(f,1) \\
&\le C_4  \limesinf{n}{\infty} (1+ \norm{f_n}_{\mathbb{M}} )^b \\
&< \infty,
\end{align*}
i.e. $f \in \mathcal{I}(\mathbb{M})$ due to part (a). Quite similar, using \eqref{eq:10101803}, we derive for any $n \in \N$ that 
\begin{equation*}
H(f_n-f,1)  \le C_4 \limesinf{m}{\infty}  \max \left \{ \norm{f_n-f_m}_{\mathbb{M}}^a , \norm{f_n-f_m}_{\mathbb{M}}^b \right \}  \le C_4 \,  2^{-an}.
\end{equation*}
In view of \eqref{eq:10101802} this shows that $\norm{f_n-f}_{\mathbb{M}} \rightarrow 0$ as $n \rightarrow \infty$. \\
For the additional statement of part (d) consider $f,f_1,... \in \mathcal{I}(\mathbb{M})$ (diffeferent from before). By linearity, L\'{e}vy's continuity theorem and after a change of variables we observe that $I(f_n) \rightarrow I(f)$ in probability if and only if
\begin{equation} \label{eq:12101804}
\forall \theta \in \R^m: \quad \left | \integral{S}{}{ \psi_{s} \left( ( f_n(s)-f(s))^{*} \theta \right)   }{\nu(ds)} \right| = \integral{S}{}{ |\psi_{s} \left( ( f_n(s)-f(s))^{*} \theta \right) |  }{\nu(ds)} \rightarrow 0
\end{equation}
holds true as $n \rightarrow \infty$. First assume that $\norm{f_n-f}_{\mathbb{M}} \rightarrow 0$ and fix $\theta \in \Gamma_m$, where we let $\lambda=\norm{\theta}_{\infty}^{-1}$. Then \eqref{eq:23111750} yields 
\begin{equation*}
\integral{S}{}{ | \psi_{s} \left( ( f_n(s)-f(s))^{*} \theta \right) |  }{\nu(ds)}  \le K_2  \integral{S}{}{ \underset{\norm{u}_{\infty} \le \lambda^{-1}}{\sup} \tau_s \left( ( f_n(s)-f(s))^{*} u \right)   }{\nu(ds)}  = K_2 H(f_n-f,\lambda)
\end{equation*}
and \eqref{eq:12101804} follows from \eqref{eq:10101802}, since $\theta \in \Gamma_m$ was arbitrary (the case $\theta=0$ is obvious). Conversely, \eqref{eq:12101804} implies that a null sequence $(\xi_n)$ is defined in virtue of
\begin{equation*}
\xi_n:= \max \left \{  \integral{S}{}{ |\psi_{s} \left( ( f_n(s)-f(s))^{*} e_j \right) |  }{\nu(ds)} : j=1,...,m \right \},
\end{equation*}
where $e_j$ is the $j$-th unit vector in $\R^m$. For $u \in \R^m$ arbitrary, write $u=(u_1,...,u_m)$. Then, using \eqref{eq:23111750}, \eqref{eq:27111701} and \eqref{eq:16081801}, we obtain constants $C_{5},C_{6}>0$ such that
\begin{align*}
H(f_n-f,1) & \le C_{5} \summe{j=1}{m}   \integral{S}{}{ \underset{\norm{u}_{\infty} \le 1}{\sup} \, \tau_s \left( ( f_n(s)-f(s))^{*} u_j e_j \right)   }{\nu(ds)} \\
& \le C_{6} \summe{j=1}{m}   \integral{S}{}{ \underset{\norm{u}_{\infty} \le 1}{\sup} \, \left| \psi_s \left( ( f_n(s)-f(s))^{*}  e_j \right) \right|   }{\nu(ds)} \\
& \le  C_{6}  \,m\, \xi_n
\end{align*}
holds true. According to \eqref{eq:10101802} again this shows that $\norm{f_n-f}_{\mathbb{M}} \rightarrow 0$.
\end{proof}
\begin{example} \label{13031801}
Assume that $\mu_s$ is $\alpha(s)$-stable with rotation-invariant spectral measure $\sigma$ for every $s \in S$, i.e. $\psi_s(u)=- c \norm{u}^{\alpha(s)}$ for some $c>0$. Then we get back the situation in \cite{falconer3} (at least for $c=m=1$) and $\mathcal{I}(\mathbb{M})$ equals the set
\begin{equation} \label{eq:16071801}
\mathcal{F}_{\alpha(\cdot)} =\mathcal{F}_{\alpha(\cdot)} (\nu, m):=\{f:S \rightarrow \Li{m} \text{ measurable }| \int_S \norm{f(s)}^{\alpha(s)} \, \nu(ds) < \infty \}
\end{equation}
due to \cref{16111705}. Moreover, we verify that $\tau_s=-c^{-1} \psi_s$. Hence $\tau_s$ is convex for $\alpha(s) \ge 1$ and subadditive for $\alpha(s) \le 1$. In both cases this implies the following:
\begin{equation*}
\forall u_1,u_2 \in \R^m \, \, \forall \eta_1,\eta_2 \ge 0 \text{ with } \eta_1+\eta_2=1: \quad \tau_s(\eta_1 u_1 + \eta_2 u_2) \le \tau_s(u_1) + \tau_s(u_2).
\end{equation*}
Using the ideas of \cite{musielak} we conclude that the mapping $\norm{f}_{\mathbb{M}}^{*}:= \inf \{\lambda>0: H(f,\lambda) \le \lambda \}$ provides the triangular inequality. However, we lose the homogeneity property in this way. 
\end{example}
\begin{remark}
In particular, if we assume that $\mu=\mu_s$ is $\alpha$-stable for fixed $0< \alpha <2$, our construction of random measures covers the considerations in \cite{bms}, \cite{lixiao} and \cite{SaTaq94}. Moreover, if  $\mu=\mu_s$ is still constant but operator-stable (say with exponent $B$), we also get back the setting of \cite{paper2}. Merely note that, in contrast to \cite{paper2}, $\mu$ has to be symmetric (instead of strictly) operator-stable in our framework. However, $\tau=\tau_s$ does not depend on $s \in S$ either and \eqref{eq:16111710} becomes
\begin{equation*}
H(f,\lambda):= \integral{S}{}{ \underset{\norm{u}_{\infty} \le \lambda^{-1}}{\sup} \tau(f(s)^{*}u)}{\nu(ds)} 
\end{equation*}
in this case. Especially, the results of \cref{16111705} remain true accordingly.
\end{remark}
But in general and in contrast to \cref{13031801} the function $\tau_s$ can not be computed explicitly. However, we are interested in possibly large subsets of $\mathcal{I}(\mathbb{M})$ that are at least similar to \eqref{eq:16071801}. For this purpose let 
\begin{equation*}
\mathcal{F}_{p} =\mathcal{F}_{p} (\nu, m):= \left \{f:S \rightarrow \Li{m} \text{ measurable }| \int_S \norm{f(s)}^{p} \, \nu(ds) < \infty \right \}, \quad p \ge 0.
\end{equation*} 
Moreover, we define $\mathcal{F}_{p,q}:= \mathcal{F}_{p} \cap \mathcal{F}_{q}$ which means for measurable $f:S \rightarrow \Li{m}$ that $f \in \mathcal{F}_{p,q}$ if and only if $\int_S \norm{f(s)}^{p,q} \, \nu(ds)< \infty$. 
\begin{cor} \label{16111703}
Let $\mathbb{M}$ be an ISRM as in \cref{04101702}. Then we have:
\begin{itemize}
\item[(a)] $\mathcal{F}_{a,b} \subset \mathcal{I}(\mathbb{M})$.
\item[(b)] Let $O(s) \in$ GL$(\R^m)$ be an orthogonal matrix such that 
\begin{equation*}
D(s)=O(s)B(s)O(s)^{*}=\text{diag}(d_1^{-1}(s),...,d_m^{-1}(s)),
\end{equation*}
i.e. $a \le d_i(s)\le b$ for any $i=1,...,m$ and $s \in S$. Define $g(s):=f(s) O(s)^{*}$, where $g(s)=(g_{i,j}(s))_{i,j=1,...,m}$ with column vectors $g^{(j)}(s) =( g_{1,j}(s),...,g_{m,j}(s))^t $. Then the following three statements are equivalent. 
\begin{itemize}
\item[(i)] $f \in \mathcal{I}(\mathbb{M})$.
\item[(ii)] $g_{i,j} \in \mathcal{F}_{d_j(\cdot)}$ for $i,j=1,...,m$.
\item[(iii)] $\int_S \norm{g^{(j)}(s)}^{d_j(s)} \, \nu(ds) < \infty$ for $j=1,...,m$.
\end{itemize}
\item[(c)] If $f_{i,j} \in \mathcal{I}(\mathbb{M})$ for $i,j=1,...,m$, then $f \in \mathcal{I}(\mathbb{M})$.
\end{itemize}
\end{cor}
\begin{proof}
Since $\norm{f(s)^{*}u} \le \norm{f(s)} \norm{u}$, part (a) follows immediately from \cref{16111705} and \eqref{eq:24111732}. Concerning part (b) note that (ii) $\Leftrightarrow$ (iii) is obvious because of $a \le d_i(s)\le b$ for any $1 \le i \le m$ and $s \in S$. In what follows we will prove that there exists some $C>0$ such that
\begin{equation} \label{eq:21091803}
 |g_{i_0,j_0}(s)|^{d_j(s)} \le \underset{\norm{u}_{\infty} \le 1}{\sup} \tau_s(f(s)^{*}u)  \le C  \summezwei{j=1}{m} \summezwei{i=1}{m} |g_{i,j}(s)|^{d_j(s)}
\end{equation}
holds true for every $1 \le i_0,j_0 \le m$ and $s \in S$. Then (i) $\Leftrightarrow$ (ii) would follow due to the fact that $f \in \mathcal{I}(\mathbb{M})$ if and only if $H(f, 1)< \infty$ (see \cref{16111705}).   \\ 
In order to prove \eqref{eq:21091803} we will first assume that $B(s)$ is already diagonal for any $s \in S$, that means $D(s)=B(s),O(s)=\mathbb{E}_m$ and $g(s)=f(s)$. In this case we see that 
\begin{equation} \label{eq:20091801}
\tau_s(\gamma \, e_j)=|\gamma|^{d_j(s)} \quad \text{for any $\gamma \in \R$ and $j=1,...,m$}.
\end{equation}
Hence, in view of \eqref{eq:23111750} and \eqref{eq:16081801}, we obtain a constant $C>0$ such that the following estimate holds for any $s \in S$, where $u=(u_1,...,u_m)$.
\begin{align}
 \underset{\norm{u}_{\infty} \le 1}{\sup} \tau_s(f(s)^{*}u)  &=  \underset{\norm{u}_{\infty} \le 1}{\sup} \tau_s \left( \summezwei{j=1}{m} \summezwei{i=1}{m} f_{i,j}(s)u_i e_j \right)  \notag \\
& \le   C   \summezwei{j=1}{m} \summezwei{i=1}{m} \underset{\norm{u}_{\infty} \le 1}{\sup}  |f_{i,j}(s)u_i|^{d_j(s)} \notag \\
& \le C    \summezwei{j=1}{m} \summezwei{i=1}{m} |f_{i,j}(s)|^{d_j(s)}.  \label{eq:21091806}
\end{align}
Conversely, consider $x=(x_1,...,x_m) \in \Gamma_m$ arbitrary. Write $l_s(x)=(l_{s,1}(x),...,l_{s,m}(x))$ and observe that $x_j e_j= l_{s,j}(x) \tau_s(x)^{B(s)}  e_j$, since $B(s)$ is diagonal. If we combine \eqref{eq:20091801} with $|l_{s,j}(x)| \le 1$, it follows (similar to Lemma 2.1 in \cite{exact}) for any $1 \le j \le m$ and $s \in S$ that 
\begin{equation} \label{eq:21091801}
\tau_s(x_j e_j) = \tau_s(x) \tau_s(l_{s,j}(x) e_j) \le \tau_s(x),
\end{equation}
which remains true for $x=0$. Let $i_0,j_0 \in \{1,...,m\}$ as well as $s \in S$ be arbitrary. Then, using \eqref{eq:20091801} for $\gamma=f_{ij}(s)$ and \eqref{eq:21091801} for $x=(f_{i,1}(s),...,f_{m,i}(s)) $, we obtain that
\begin{equation} \label{eq:21091810}
|f_{i,j}(s)|^{d_j(s)} = \tau_s(f_{i,j}(s) e_j) \le  \tau_s (x) =\tau_s (f(s)^{*} e_i) \le \underset{\norm{u}_{\infty} \le 1}{\sup} \tau_{s}(f(s)^{*}u).
\end{equation}
This completes the proof of part (b) for the diagonal case, since \eqref{eq:21091803} follows from \eqref{eq:21091806} and \eqref{eq:21091810} with $g(s)=f(s)$. In the general case, where $D(s)=O(s)B(s)O(s)^{*}$, we can write
\begin{equation*}
O(s)x=O(s) \tau_s(x)^{B(s)} O(s)^{*} O(s) l_s(x)=\tau_s(x)^{D(s)} O(s) l_s(x)
\end{equation*}
for any $s \in S$ and $x \in \Gamma_m$ (see Proposition 2.2.2 in \cite{thebook}). In view of $O(s) l_s(x)  \in S^{m-1}$ it follows that $\tau_s(x)=\tau_{D(s)}(O(s)x)$ by uniqueness (also recall \cref{17111701}), where the case $x=0$ is included again. Hence we derive that
\begin{equation*}
\forall s \in S : \quad \underset{\norm{u}_{\infty} \le 1}{\sup} \tau_{s}(f(s)^{*}u)= \underset{\norm{u}_{\infty} \le 1}{\sup} \tau_{D(s)}(g(s)^{*}u)
\end{equation*}
which allows us to argue as before to establish \eqref{eq:21091803} for the function $g(s)=f(s)O(s)^{*}$. \\
We now prove part (c) and assume that $f_{i,j} \mathbb{E}_m \in \mathcal{I}(\mathbb{M})$ for $i,j=1,...,m$. Using the implication (i) $\Rightarrow$ (iii) of part (b) accordingly this shows that $f_{i,j} \in \mathcal{F}_{d_k(\cdot)}$ for $i,j = 1,...,m$ and $k=1,...,m$, since the columns of $O(s)^{*}$ belong to $S^{m-1}$. Particularly, the components of $O(s)^{*}$ are bounded by one for every $s \in S$. Using this we can verify that condition (ii) of part (b) holds true with respect to the function $f$ itself, i.e. $f \in \mathcal{I}(\mathbb{M})$. 
\end{proof}
\begin{example}  \label{09031801}
Consider $m=2$ and let $|\cdot|$ be the Lebesgue measure on $S=\R$. Moreover, let $B(s)=\text{diag}(1+\exp(-|s|),2)$. Then the function 
\begin{equation*}
f(s)=\text{diag}(|s|^{-3/2(1+\exp(-|s|) },|s|^{-3}) \indikatorzwei{|s|>1}, \quad s \in \R
\end{equation*}
fulfills condition (ii) of \cref{16111703} which leads to $f \in \mathcal{I}(\mathbb{M})$. Now let $g(s)=f_{1,1}(s) \mathbb{E}_2$ and observe that 
\begin{equation*}
-\frac{3}{4}(1+\exp(-|s|)) \ge - \frac{15}{16} \quad  \text{for any $|s| \ge  \ln 4$}.
\end{equation*}
This implies that 
\begin{equation*}
\integral{\R}{}{|g_{2,2}(s)|^{1/2}}{ds}=\integral{\R}{}{|f_{1,1}(s)|^{1/2}}{ds}=\infty,
\end{equation*}
so $f_{1,1}(s)$ does not belong to $ \mathcal{I}(\mathbb{M})$. Particularly, the converse of part (c) is false. In a similar way we see that $ \mathcal{F}_{a,b} $ is a true subset of $ \mathcal{I}(\mathbb{M})$ in general.
\end{example}
To finish this section we want to find a suitable generalization of Theorem 2.7 in \cite{falconer3}. For this purpose we consider $(S,\Sigma)=(\R^d, \B(\R^d))$ with $|\cdot|$ being the d-dimensional Lebesgue measure, where $d \ge 1$ is arbitrary. Also define the \textit{scaling map} $T_{u,r}:\R^d \rightarrow \R^d$ via $T_{u,r}(s)=(s-u)/ r$ for any $u \in \R^d$ and $r>0$. Then \cref{16111705} implies that the following integral exists at least for every $f \in \mathcal{F}_{a,b}$.
\begin{equation} \label{eq:08011802}
\integral{\R^d}{}{f(s)}{(T_{u,r} \mathbb{M})(ds)}:=\integral{\R^d}{}{f  (T_{u,r}(s))}{ \mathbb{M}(ds)}.
\end{equation}
Hence $(T_{u,r} \mathbb{M})$ is some kind of an image measure. We thereby justify the notion \textit{multi operator-stable} ISRM which means:  Locally, say \textit{around} $u$, the ISRM $\mathbb{M}$ somehow behaves like an operator-stable one with exponent $B(u)$, denoted by $\mathbb{M}_{u}$ in sequel (see \cite{paper2}).
\begin{theorem} \label{09011805}
Let $\mathbb{M}$ be a multi operator-stable ISRM on $\mathcal{S}=\{A \in \B(\R^d): |A|< \infty\}$. Fix $u \in \R^d$ and assume that the mapping $\R^d \ni s \mapsto B(s)$ is continuous at $u$ such that 
\begin{equation}
\limes{r}{0} r^{dB(u+rs)} r^{-dB(u)} =\mathbb{E}_m  \label{eq:08011801}
\end{equation} 
holds uniformly in $s$ on compact subsets of $\R^d$. Moreover, consider $f_1,...,f_n  \in \mathcal{F}_{a,b} $ with compact support such that, for any $j=1,...,n$ and $s_1,s_2 \in \R^d$, we have that
\begin{equation} \label{eq:08011803}
f_j(s_1) B(s_2)=B(s_2) f_j (s_1). 
\end{equation}
Then, as $r \rightarrow 0$, the following convergence in distribution holds true:
\begin{equation*}
\left( r^{-d B(u)} \integral{\R^d}{}{f_j(s)}{ (T_{u,r} \mathbb{M})(ds)}  : j=1,...,n \right) \Rightarrow \left( \, \, \integral{\R^d}{}{f_j(s)}{ \mathbb{M}_{u}(ds) }  : j=1,...,n \right),
\end{equation*}
where $\mathbb{M}_u \sim (B(u),\sigma, \nu)$ in the sense of \cref{04101702}. Especially, we have
\begin{equation*}
(r^{-dB(u)}\mathbb{M}(u+r A_j),...,r^{-d B(u)}\mathbb{M}(u+rA_n)) \Rightarrow (\mathbb{M}_u(A_1),...,\mathbb{M}_u(A_n))
\end{equation*}
for any bounded sets $A_1,...,A_n \in \B(\R^d)$.
\end{theorem}
\begin{proof}
Without loss of generality we may assume that $f_1,...,f_n$ have a support contained in $[-K,K]^d$ for some $K>0$. Fix $\theta_1,...,\theta_n \in \R^{m}$. Then, in view of \eqref{eq:17111740}, \eqref{eq:08011802} and L\'{e}vy's continuity theorem it suffices to show that
\begin{equation} \label{eq:09011801}
\integral{\R^d}{}{\psi_s \left( \summe{j=1}{n} f_j \left (\frac{s-u}{r}\right)^{*} r^{-d B(u)} \theta_j \right)}{ds} \rightarrow \integral{\R^d}{}{ \psi_u \left( \summe{j=1}{n} f_j(s)^{*} \theta_j \right)}{ds}
\end{equation}
as $r \rightarrow 0$. After a change of variables the left-hand side of \eqref{eq:09011801} can be rewritten as
\begin{equation*}
\integral{\R^d}{}{ r^d \, \psi_{u+rs} \left( \summe{j=1}{n} f_j \left (s \right)^{*} r^{-d B(u)} \theta_j \right) }{ds} = \integral{\R^d}{}{ \psi_{u+rs} \left( \summe{j=1}{n} f_j \left (s \right)^{*} r^{dB(u+rs)}r^{-d B(u)} \theta_j \right) }{ds},
\end{equation*}
where we used \eqref{eq:08011803} and the $B(u+rs)$-homogeneity of $\psi_{u+rs}$. As implicitly seen in the proof of \cref{23111710} the continuity of $s \mapsto B(s)$ at $u$ ensures that $\mu_{u+rs} \rightarrow \mu_u$ weakly for any $s \in \R^d$ as $r \rightarrow 0$ and hence that $\psi_{u+rs} (\cdot) \rightarrow \psi_u (\cdot)$ uniformly on compact sets (see Lemma 3.1.10 in \cite{thebook}). Since \eqref{eq:08011801} particularly implies that $r^{dB(u+rs)}r^{-dB(u)}\rightarrow \mathbb{E}_m$, we derive the following convergence for any $s \in \R^d$ as $r \rightarrow 0$.
\begin{equation} \label{eq:09011803}
\psi_{u+rs} \left( \summe{j=1}{n} f_j \left (s \right)^{*} r^{dB(u+rs)}r^{-d B(u)} \theta_j \right) \rightarrow  \psi_u \left( \summe{j=1}{n} f_j(s)^{*} \theta_j \right).
\end{equation}
Note that the left-hand side of \eqref{eq:09011803} vanishes for $s \notin [-K,K]^d$, while $\norm{r^{dB(u+rs)}r^{-d B(u)}}$ is bounded for $r>0$ sufficiently small and $s \in [-K,K]^d$ due to \eqref{eq:08011801}. Using \eqref{eq:24111732}, \eqref{eq:23111750} and $\norm{f(s)}=\norm{f(s)^{*}}$ it follows that there exists a constant $C>0$ fulfilling
\begin{equation*}
\left| \psi_{u+rs} \left( \summe{j=1}{n} f_j \left (s \right)^{*} r^{dB(u+rs)}r^{-d B(u)} \theta_j \right) \right|  \le C \summe{j=1}{n} \max \{1,\norm{\theta_j}^b\} \norm{f_j(s)}^{a,b} 
\end{equation*}
for any $s \in \R^d$. In view of $f_1,...,f_n  \in \mathcal{F}_{a,b} $ this gives \eqref{eq:09011801} by dominated convergence.
\end{proof}
Since the previous assumptions might appear challenging we state some useful observations. 
\begin{remark} \label{11071802}
\begin{itemize}
\item[(a)] Consider $\chi(r,s)=r^{dB(u+rs)} r^{-dB(u)}$. Then the proof above revealed that the assumptions of \cref{09011805} can be softened in two ways: On one hand the convergence in \eqref{eq:08011801} merely needs to hold pointwise as long as $\chi(r,s)$ is bounded for any $r>0$ (small) and $s \in K$, where $K \subset \R^d$ is compact. On the other hand $f_1,...,f_n$ do not need to have a compact support, if $\chi(r,s)$ is even bounded for any $r>0$ (small) and $s \in \R^d$. 
\item[(b)] If $B(s)= \alpha(s)^{-1} \mathbb{E}_m$ for any $s \in \R^d$ with $\alpha: \R^d \rightarrow [a,b] \subset (0,2)$, we call $\mathbb{M}$ a (multivariate) \textit{multi-stable} ISRM and compute that
\begin{equation*} 
 r^{dB(u+rs)} r^{-dB(u)}=\exp \left( d \frac{ (\ln r) (\alpha(u)-\alpha(u+rs)) }{\alpha (u) \alpha(u+rs)}  \right) \mathbb{E}_m.
\end{equation*}
Hence in this case a sufficient condition for \eqref{eq:08011801} is given by 
\begin{equation} \label{eq:25071802} 
|\alpha(u+s) - \alpha(u)| = o (1/ \ln \norm{s}) \quad (s \rightarrow 0)
\end{equation}
which is weaker than the corresponding assumption (2.14) in \cite{falconer3}. Also note that \eqref{eq:25071802} is always fulfilled, if $\alpha(\cdot)$ locally satisfies a $\gamma$-H\"older-condition for some $0 < \gamma \le 1$.
\item[(c)] Assumption \eqref{eq:08011803} is superfluous for $f:\R^d \rightarrow \R$ and for general $f: \R^d \rightarrow \R^m$, if $\mathbb{M}$ is a multi-stable ISRM (see part (b) and also the more specific situation in \cref{13031801}), respectively. 
\end{itemize}
\end{remark}
\section{Moving-average representation and further examples} \label{chapter3}
As in \cref{09011805} and throughout this section let $\mathbb{M}$ be a multi operator-stable ISRM on $\mathcal{S}=\{A \in \B(\R^d): |A|< \infty\}$. In this section we want to analyze multivariate random fields $\{X(t):t \in \R^d\}$, where each $X(t)$ is an $\R^m$-valued random vector defined by a random integral with respect to $\mathbb{M}$ of suitable functions. Hence the \textit{domain} of these random fields is $\R^d$ and their \textit{state space} is $\R^m$. We start with an extension of Proposition 3.1 in \cite{falconer3}. For this purpose recall the definition of $\norm{\cdot}_{\mathbb{M}}$ from \cref{16111705}.
\begin{prop} \label{27041805}
Let $\mathbb{X}=\{X(t):t \in \R^d\}$ be an $\R^m$-valued random field such that $X(t)=I_{\mathbb{M}}(f(t,\cdot))$ for every $t \in \R^d$, where $f(t,\cdot) \in \mathcal{I}(\mathbb{M})$. Furthermore, suppose that there exists some $\xi > d / a$ such that, for any $K>0$, we can find a constant $L>0$ fulfilling
\begin{equation} \label{eq:09021803}
\norm{f(t_1,\cdot)-f(t_2,\cdot)}_{\mathbb{M}} \le L \norm{t_1-t_2}^{\xi} \quad \text{for all $t_1,t_2 \in [-K,K]^d$}.
\end{equation}
Then, for every $0<\gamma< \xi - d/a$, $\mathbb{X}$ has a version that is locally $\gamma$-H\"older-continuous. Particularly, \eqref{eq:09021803} is fulfilled, if 
\begin{equation} \label{eq:12021801} 
\integral{\R^d}{}{  \norm{f(t_1,s)-f(t_2,s)}^{ 1/ \lambda(s),  1/ \Lambda(s)}  }{ ds} \le \tilde{L} \norm{t_1-t_2}^{\xi b} 
\end{equation}
holds true for all $t_1,t_2 \in [-K,K]^d$ and for some suitable $\tilde{L}>0$ (depending on $K$).
\end{prop}
\begin{proof}
Fix $K>0$ and $d/ \xi<p<a$ which is possible due to our assumption. Then by linearity, \cref{16111705} (c) and \eqref{eq:09021803} we see that there exists a $C>0$ such that 
\begin{equation*}
\erwartung{\norm{X(t_1)-X(t_2)}^p} \le C \norm{t_1-t_2}^{p \xi}  \quad \text{for all $t_1,t_2 \in [-K,K]^d$}.
\end{equation*}
Note that $p \xi-d >0$. Hence the multivariate version of Kolmogoroff's continuity theorem (see Remark 21.7 in \cite{Kle08}) gives the assertion, since $(p \xi-d)/p \rightarrow \xi - d/a$ as $p \rightarrow a$. \\
For the additional statement fix $t_1,t_2 \in [-K,K]^d$ with $t_1 \ne t_2$ and let $\lambda_0:= T \norm{t_1-t_2}^{\xi}$, where $T>0$ is arbitrary. Also recall \eqref{eq:16111710}. Then, in view of \cref{24111750} and by equivalence of norms, we find constants $C_1,C_2>0$ such that 
\allowdisplaybreaks
\begin{align*}
 H(f(t_1,\cdot)-f(t_2,\cdot), \lambda_0) &  \le  C_1   \integral{\R^d}{}{  \underset{\norm{u}_{\infty} \le \lambda_0^{-1}}{\sup}   \norm{ (f(t_1,s)-f(t_2,s))^{*} u}^{ 1/ \lambda(s),  1/ \Lambda(s)}   }{ ds} \\
& \le C_2 \, \tilde{L} \, \norm{t_1-t_2}^{\xi b} \max \{\lambda_0^{-a} , \lambda_0^{-b} \}  \\
& = C_2 \, \tilde{L}  \, \max \{ T^{-a} \norm{t_1-t_2}^{\xi (b-a)} ,T^{-b} \}
\end{align*}
holds true by \eqref{eq:12021801}. Hence, for $T$ sufficiently large (only depending on $K$ and $\tilde{L}$), the last expression is bounded by one which implies \eqref{eq:09021803}. 
\end{proof}
Now we want to proceed with the construction suggested above. More precisely, by an appropriate choice of integrands $f(t,\cdot)$ we will obtain a rich class of random fields that mostly generalize several so called moving-average representations as we will illustrate in \cref{14031811} below. Furthermore, there is a quite natural relation between all these random fields that will be analyzed in \cref{chapter4}. \\ \quad \\
In order to do so we need some technical assumptions. Let us recall that a function $\phi:\R^d \rightarrow [0,\infty)$ is called $E$-homogeneous for some operator $E \in Q(\R^d)$, if $\phi(r^E x)=r \phi(x)$ for any $r>0$ and $x \in \Gamma_d$. On the other hand a continuous function $\phi:\R^d \rightarrow [0,\infty)$ is called \textit{$(\beta,E)$-admissible}, where $\beta>0$, if the following properties are fulfilled. 
\begin{itemize}
\item[(i)] $\phi(x)>0$ for every $x \in \Gamma_d$.
\item[(ii)] For any $0<A \le B$ there exists a constant $C>0$ such that the implication 
\begin{equation*}
\tau_E(x) \le 1 \quad \Rightarrow \quad |\phi(x+y)-\phi(y)| \le C \tau_E(x)^{\beta}
\end{equation*}
holds true, whenever $A \le \norm{y} \le B$.
\end{itemize}
Often both properties are combined. For more details and examples see \cite{bms}.
\begin{theorem} \label{02031801}
Let $E \in Q(\R^d)$ with $q=tr(E)$ and $D(s) \in \Li{m}$ being symmetric for every $s \in \R^d$. Moreover, let $\phi:\R^d \rightarrow [0,\infty)$ be an $E$-homogeneous and $(\beta,E)$-admissible function for some $\beta>0$. Finally assume that either condition $(C1)$ or $(C2)$ is fulfilled, namely:
\begin{itemize}
\item[\textbf{(C1)}] We have that
\begin{equation} \label{eq:22031801}
-q / b< \inf \nolimits_{s \in \R^d} \lambda_{D(s)-qB(s)} \le  \sup \nolimits_{s \in \R^d} \Lambda_{D(s)-qB(s)}<\beta -q / a.
\end{equation}
\item [\textbf{(C2)}] For almost every $s \in \R^d$ we have that $B(s)D(s)=D(s)B(s)$ together with
\begin{equation} \label{eq:22031802}
0< \inf \nolimits_{s \in \R^d} \lambda_{D(s)} \le \sup \nolimits_{s \in \R^d} \Lambda_{D(s)}< \beta.
\end{equation}
\end{itemize}
Then the stochastic integral 
\begin{equation} \label{eq:01021801}
X(t):=\integral{\R^d}{}{ \left[ \phi(t-s)^{D(s)-qB(s)}-\phi(-s)^{D(s)-qB(s)} \right]}{\mathbb{M}(ds)}
\end{equation}
exists for every $t \in \R^d$. The resulting $\R^m$-valued random field $\mathbb{X}=\{X(t):t \in \R^d\}$ is called a $(\phi,D(s))$-moving-average representation (with respect to $\mathbb{M}$).
\end{theorem}
\begin{proof}
Obviously $X(0)=0$ almost surely. Then, for fixed $t \in \Gamma_d$, we have to show that
\begin{equation*}
f(t,s):=  \phi(t-s)^{D(s)-qB(s)}-\phi(-s)^{D(s)-qB(s)}, \quad s \in \R^d
\end{equation*}
belongs to $\mathcal{I}(\mathbb{M})$. According to \cref{16111701} the definition of $f(t,\cdot)$ can be neglected on the Lebesgue null set $\{0,t \}$. Then, in view of the proof of \cref{09021801} below, it is useful to consider a much more general case first, i.e. let $A(s) \in \Li{m}$ be symmetric such that $A(s)B(s)=B(s)A(s)$ holds for any $s \in \R^d$ (or at least almost everywhere, see above). Moreover, for fixed $ \gamma \in [0,q]$ we assume that 
\begin{equation} \label{eq:19031801}
\frac{\gamma-q}{b} < \rho_1 := \inf \nolimits_{s \in \R^d} \lambda_{A(s)} \le \sup \nolimits_{s \in \R^d} \Lambda_{A(s)}=:\rho_2 < \beta + \frac{\gamma-q}{a}
\end{equation}
holds true and define for a fixed $t_0 \in \R^d$ the function
\begin{equation} \label{eq:16081805}
g(t,t_0,s):=  \phi(t-s)^{A(s+t_0)-\gamma B(s+t_0)}-\phi(-s)^{A(s+t_0)- \gamma B(s+t_0)}, \quad s \in \R^d.
\end{equation}
Then, for fixed $\lambda>0$, we will show that 
\begin{equation} \label{eq:31071801}
\integral{\R^d}{}{ \underset{\norm{u}_{\infty} \le \lambda^{-1} }{\sup} \tau_{s+t_0}(g(t,t_0,s)u) }{ds} < \infty.
\end{equation}
For this purpose let $(\tau(\cdot),l(\cdot))$ be the generalized polar coordinates with respect to $E$, where $\tau(0):=0$, and define the bounded sets $F(\kappa):=\{s: \tau(s) \le \kappa\}$ for any $\kappa>0$. Then our assumptions on $\phi$ imply that $\phi(s)=\tau(s) \phi(l(s))$ for any $s \in \R^d$, where
\begin{equation} \label{eq:26031810}
0<m_{\phi}:=\min_{\theta \in S_E} \phi(\theta) \le \phi(l(s)) \le M_{\phi}:=\max_{\theta \in S_E} \phi(\theta).
\end{equation}
Moreover, since $A(s)$ is symmetric, \eqref{eq:15111704} holds accordingly for $A(s)$ (even if $\rho_1 \le \rho_2 \le 0$) and in view of \eqref{eq:19031801} we verify that $\norm{\xi^{A(s)}}$ is bounded for any $s \in \R^d$ and $m_{\phi} \le \xi \le M_{\phi}$. Hence, for fixed $\eta>0$ and due to \eqref{eq:15111704} again, there exists a constant $C_1 =C_1(\eta) \ge 1$ with
\begin{equation} \label{eq:24011802}
\forall s \in F(\eta): \quad \norm{\phi(s)^{A(s+t_0 )}} \le \norm{\xi^{A(s+t_0)}} \, \norm{\tau(s)^{A(s+t_0)}} \le C_1 \tau(s)^{\lambda_{A(s+t_0)} }.
\end{equation}
Additionally, we have already seen that $\tau_{s}(x+y) \le C_2(\tau_{s}(x)+\tau_{s}(y))$ for any $s,x,y \in \R^d$ and some $C_2>0$. Also recall that $A(s)B(s)=B(s)A(s)$. Therefore the $B(s)$-homogeneity of $\tau_s$, \eqref{eq:26031810} and $\tau_s(-x)=\tau_s(x)$ imply for any $s \in F(\eta)$ that
\begin{align}  
 \sup_{ \norm{u}_{\infty} \le  \lambda^{-1}} \tau_{s+t_0}(g(t,t_0,s)u) &  \le C_2 m_{\phi}^{-\gamma} \left (\tau(t-s)^{-\gamma} \underset{\norm{u}_{\infty} \le  \lambda^{-1}}{\sup} \tau_{s+t_0}(  \phi(t-s)^{A(s+t_0)} u) \right. \notag \\
& \left. \qquad \qquad \qquad \qquad + \tau(s)^{-\gamma} \underset{\norm{u}_{\infty} \le  \lambda^{-1}}{\sup} \tau_{s+t_0}(  \phi(-s)^{A(s+t_0)} u)  \right ) . \label{eq:26031833}
\end{align} 
Now we use the estimates for $\tau_s$ from \eqref{eq:24111732} combined with \eqref{eq:24011802}. Then the following computation holds for every $s \in F(\eta)$ and some constant $C_3=C_3(\lambda)>0$, since $0< a \le b$.
\allowdisplaybreaks
\begin{align}
\underset{\norm{u}_{\infty} \le  \lambda^{-1}}{\sup} \tau_{s+t_0}(  \phi(-s)^{A(s+t_0)} u)  & \le C_1^b  C_3 \left(  \tau(s)^{a \lambda_{A(s+t_0)}}+ \tau(s)^{b \lambda_{A(s+t_0)} } \right) \notag \\
& \le C_1^b  C_3  \max \{1, \eta^{b(\rho_2-\rho_1)}  \} \left(  \tau(s)^{a \rho_1}+ \tau(s)^{b \rho_1 } \right) \label{eq:22031820}.
\end{align}
Recall \eqref{eq:19031801} and that $\gamma-q \le 0$. Using Proposition 2.3 in \cite{bms} this easily yields that 
\begin{equation} \label{eq:26031830}
\integral{F(\eta)}{}{ ( \tau(s)^{a\rho_1-\gamma}+\tau(s)^{b\rho_1 -\gamma} ) }{ds} < \infty.
\end{equation}
Moreover, we know that $\tau(x+y) \le C_4(\tau(x)+\tau(y))$ holds for any $x,y \in \R^d$ and some $C_4 \ge 1$ due to Lemma 2.2 in \cite{bms}. This implies that $F(\eta) \subset \{s: \tau(t-s) \le C_4(\tau(t)+\eta) \}$. Now argue similar as in \eqref{eq:22031820} to verify that there exists some $C_5=C_5(\lambda,\eta,\tau(t))>0$ fulfilling
\begin{align*}
& \left( \underset{\norm{u}_{\infty}\le  \lambda^{-1}}{\sup} \tau_{s+t_0}(  \phi(t-s)^{A(s+t_0)} u) \right) \indikator{F(\eta)}{s} \\
& \qquad \le C_5 \left( \tau(t-s)^{a \rho_1}  +  \tau(t-s)^{b \rho_1}  \right) \indikator{F(C_4(\tau(t)+\eta))}{t-s} .
\end{align*}
Hence by a change of variables this mainly reduces the problem to \eqref{eq:26031830}. Overall and using the bounds from \eqref{eq:26031833} this shows that 
\begin{equation} \label{eq:02081801}
\integral{F(\eta)}{}{ \underset{\norm{u}_{\infty} \le \lambda^{-1} }{\sup} \tau_{s+t_0}(g(t,t_0,s)u) }{ds} < \infty.
\end{equation}
Henceforth we consider the behavior of $g(t,t_0,s)$ for $s \in F(\eta)^c=\{s:\tau(s)> \eta \}$. First of all the assumptions on $\phi$ allow to argue as in (4.5) and (4.6) of \cite{paper} (also see the proof of Theorem 2.5 in \cite{lixiao}). Thus we obtain a constant $C_6>0$ such that the following implication holds for any $s \in \Gamma_d$.
\begin{equation} \label{eq:25011820}
\tau(x) \le 1 \quad \Rightarrow \quad |\phi(x+ \phi(s)^{-E}s) - \phi( \phi(s)^{-E}s ) |= |\phi(x+ \phi(s)^{-E}s) - 1 | \le C_6 \tau(x)^{\beta}.
\end{equation} 
Now we specify and fix the choice of $\eta>0$ in such a way that each of the following inequalities holds, whenever $\tau(s)> \eta$. Note that this is possible as shown in eq. (4.7) of \cite{paper2}.
\begin{equation}  \label{eq:25011810}
\phi(s)^{-1} \tau(t)<1, \quad C_6 \tau(t)^{\beta}\phi(-s)^{- \beta}< \frac{1}{2}, \quad  \phi(-s)>1.
\end{equation}
Then, using the joint continuity of the exponential operator (see Proposition 2.2.11 in \cite{thebook}), we can argue similar as in the proof of Theorem 2.5 in \cite{lixiao}  to verify for any $s \in \R^d$ and $1/2 < u < 3/2$ that
\begin{equation} \label{eq:26011802}
\norm{u^{A(s)-\gamma B(s)} - \mathbb{E}_m}  \le \norm{A(s)-\gamma B(s)} \sup_{1/2 \le r \le 3/2} \norm{r^{A(s)- \gamma B(s)- \mathbb{E}_m} } \, |u-1| \le C_7 |u-1|.
\end{equation}
Here $C_7>0$ is a constant. Moreover, we used that $A(s)$ and $B(s)$ are symmetric such that
\begin{equation*} 
\forall s \in \R^d: \quad \norm{A(s)- \gamma B(s)} \le \norm{A(s)}+\gamma \norm{B(s)}  \le |\rho_1|+|\rho_2| + \gamma /a.
\end{equation*}
Moreover, the $E$-homogeneity of $\phi$ and the assumptions on $A(s)$ imply for any $s \in \Gamma_d$ that
\begin{align} 
& \phi(t-s)^{A(s+t_0 )-\gamma B(s+t_0)}-\phi(-s)^{A(s+t_0)- \gamma B(s+t_0)} \notag  \\
& \qquad = (\phi(\phi(-s)^{-E}t-\phi(-s)^{-E}s)^{A(s+t_0)-\gamma B(s+t_0)}  -\mathbb{E}_m ) \phi(-s)^{A(s+t_0)- \gamma B(s+t_0)} \notag  \\
& \qquad =  \phi(-s)^{ -\gamma B(s+t_0)}   \phi(-s)^{A(s+t_0)}  (\phi(\phi(-s)^{-E}t-\phi(-s)^{-E}s)^{A(s+t_0)-\gamma B(s+t_0)}  -\mathbb{E}_m )  \label{eq:26011801} .
\end{align}
Recall that $\tau(\cdot)$ is an $E$-homogeneous function and consider $x(s,t):=\phi(-s)^{-E}t$, where $\tau(x(s,t))<1$ due to \eqref{eq:25011810}. If we apply \eqref{eq:25011820} to $-s$ and use \eqref{eq:25011810} once more, it follows for every $s \in F(\eta)^c$ that
\begin{equation} \label{eq:26011804}
|\phi(\phi(-s)^{-E}t- \phi(-s)^{-E}s) - 1 | \le C_6 \tau(\phi(-s)^{-E}t)^{\beta}=C_6  \tau(t)^{\beta} \phi(-s)^{-\beta} < 1/2.
\end{equation}
Hence we are allowed to use \eqref{eq:26011802} for $u(s,t):=\phi(\phi(-s)^{-E}t-\phi(-s)^{-E}s)$. Moreover, combine this with \eqref{eq:26011804} to derive for any $s \in F(\eta)^c$ that
\begin{equation}
\norm{  \phi(\phi(-s)^{-E}t-\phi(-s)^{-E}s)^{A(s+t_0)-\gamma B(s+t_0)}  -\mathbb{E}_m}  \le C_6 C_7 \tau(t)^{\beta}  \phi(-s)^{-\beta} \label{eq:21031801}.
\end{equation}
At the same time \eqref{eq:15111704}  and \eqref{eq:25011810} yield the existence of a constant $C_8>0$ which fulfills
\begin{equation} \label{eq:21031802}
\norm{\phi(-s)^{A(s+t_0)}} \le C_8 \phi(-s)^{\Lambda_{A(s+t_0)}}, \quad s \in F(\eta)^c.
\end{equation}
Finally, if we combine \eqref{eq:26011801}, \eqref{eq:21031801} and \eqref{eq:21031802}, this implies similar to \eqref{eq:26031833}-\eqref{eq:22031820} for any $s \in F(\eta)^c$ that
\allowdisplaybreaks
\begin{align*}
& \underset{\norm{u}_{\infty} \le \lambda^{-1}}{\sup} \tau_{s+t_0}(g(t,t_0,s)u) \\
& \qquad \le C_3 \phi(-s)^{- \gamma}   \norm{\phi(-s)^{A(s+t_0)}  (\phi(\phi(-s)^{-E}t-\phi(-s)^{-E}s)^{A(s+t_0)-\gamma B(s+t_0)}  -\mathbb{E}_m )}^{a,b}  \\
 & \qquad \le C_3 C_9  \left (\tau(t)^{a \beta} \phi(-s)^{a( \Lambda_{A(s+t_0)} - \beta) - \gamma   }+ \tau(t)^{b \beta} \phi(-s)^{b (\Lambda_{A(s+t_0)}- \beta) - \gamma }  \right) \\
  & \qquad \le 2 C_3 C_9  \max \{\tau(t)^{a \beta} ,\tau(t)^{b \beta} \}  (m_{\phi} \tau(s))^{a (\rho_2 - \beta) - \gamma } ,
\end{align*}
since $\phi(s)>1$, for some $C_9>0$. Throughout and in view of \eqref{eq:19031801} we used that 
\begin{equation*}
 \Lambda_{A(s+t_0)}- \beta \le \rho_2- \beta  < \frac{\gamma-q}{a} \le 0, \quad s \in \R^d.
\end{equation*}
Particularly, we have that $\int_{F(\eta)^{c}} \tau(s)^{a (\rho_2 - \beta) - \gamma } \, ds < \infty$ (see Proposition 2.3 in \cite{bms} again) and \eqref{eq:31071801} follows due to \eqref{eq:02081801}. Since $t_0 \in \R^d$ and $\lambda>0$ were arbitrary, we especially derive from \eqref{eq:31071801} that $H(g(t,0,\cdot)^{*},\lambda)=H(g(t,0,\cdot),\lambda)< \infty$ for any $\lambda>0$ and therefore that $g(t,0,\cdot) \in \mathcal{I}(\mathbb{M})$ according to \cref{16111705}. Moreover, the assertion easily follows for $f(t, \cdot)$. To verify this merely let $A(s)=D(s)-qB(s)$ with $\gamma=0$ in the situation of $(C 1)$ and $A(s)=D(s)$ with $\gamma=q$ in the situation of $(C 2)$, respectively. 
\end{proof}
Due to our approach in the previous proof we can state the following observation. 
\begin{remark} \label{28031805}
Recall the previous notation. For $\eta,\zeta>0$ and $t \in \R^d$ arbitrary we define the function $h_{\eta,\zeta}(t,\cdot):\R^d \rightarrow [0,\infty)$, where $h_{\eta,\zeta}(t,0)=h_{\eta,\zeta}(t,t)=0$ and 
 \allowdisplaybreaks
\begin{align}
h_{\eta,\zeta}(t,s) &= \left(\tau(s)^{a \rho_1 - \gamma} + \tau(s)^{b \rho_1 - \gamma}  \right) \indikator{F(C_4(\zeta +\eta))}{s}  \notag \\
& \qquad + \left( \tau(t-s)^{a \rho_1- \gamma}+ \tau(t-s)^{b \rho_1- \gamma}  \right) \indikator{F(C_4(\zeta+\eta))}{t-s} \notag \\
&\qquad + \max \{\tau(t)^{a \beta} ,\tau(t)^{b \beta} \}  \tau(s)^{a (\rho_2 - \beta) - \gamma } \indikator{F(\eta)^c}{s} \label{eq:09041801}
 \end{align}
else. Then the proof of \cref{02031801} revealed two aspects. On one hand $h_{\eta,\zeta}(t,\cdot)$ is integrable for any $\eta,\zeta>0$ and $t \in \R^d$ due to \eqref{eq:19031801}. On the other hand, for any $\lambda>0$ and $t \in \R^d$, there exists a constant $C=C(\lambda, \eta, \tau(t))>0$ such that we have
\begin{equation} \label{eq:30031801}
\underset{\norm{u}_{\infty} \le  \lambda^{-1}}{\sup} \tau_{s+t_0}(g(t,t_0,s)u)  \le C h_{\eta,\tau(t)}(t,s) \quad \text{for any $s, t_0 \in \R^d$}
\end{equation}
as long as $\eta=\eta(\tau(t))$ is chosen as in \eqref{eq:25011810}. More precisely, if $\tau(t)$ is bounded, then $\eta$ as well as $C$ can be chosen sufficiently large such that \eqref{eq:30031801} holds uniformly for those $t \in \R^d$.
\end{remark}
According to the present \textit{multi} operator-stable point of view and in contrast to Theorem 4.2 in \cite{paper} it should appear natural that $\mathbb{X}$ will not have stationary increments in general. However, we get the following properties. Here we call an $\R^m$-valued random field $\{X(t):t \in \R^d\}$ \textit{full}, if the distribution of $X(t)$ is full for any $t \in \Gamma_d$.
\begin{prop} \label{09021801}
Suppose that the assumptions of \cref{02031801} are fulfilled and let $\mathbb{X}=\{X(t):t \in \R^d\}$ be the random field defined by \eqref{eq:01021801}. Then we have:
\begin{itemize}
\item[(a)] $\mathbb{X}$ is stochastically continuous.
\item[(b)] I there exists a Borel set $A \subset \R^d$ with $|A|>0$ and such that $D(s)-qB(s) \in \text{GL}(\R^m)$ for every $s \in A$, then $\mathbb{X}$ is full.
\end{itemize}
\end{prop}
\begin{proof}
For part (a) we fix $t_0 \in \R^d$ and have to show that $X(t_0+t_n) \rightarrow X(t_0)$ in probability as $n \rightarrow \infty$, where $(t_n) \subset \R^d$ is an arbitrary null sequence. Particularly, there exists  some $\zeta>0$ such that $\tau(t_n ) \le \zeta$ for all $n \in \N$. Also fix $u \in \Gamma_m$ (the case $u=0$ is obvious) and recall \eqref{eq:17111740}. Hence by linearity, L\'{e}vy's continuity theorem and after a change of variables it suffices to prove that
\begin{equation} \label{eq:06021811}
\integral{\R^d}{}{ \psi_{s+t_0} \left( (\phi(t_n-s)^{D(s+t_0)-qB(s+t_0)}-\phi(-s)^{D(s+t_0)-qB(s+t_0)} ) u \right)   }{ds} \rightarrow 0.
\end{equation}
Using the notation from the proof of \cref{02031801} the integrand in \eqref{eq:06021811} can be rewritten as $\psi_{s+t_0} (g(t_n,t_0,s) u)$.Then, by continuity of $\psi_{s+t_0}$ and $\phi$, we see for any $s \in \Gamma_d$ that 
\begin{equation*}
\psi_{s+t_0}(g(t_n,t_0,s)u) \rightarrow \psi_{s+t_0}(0)=0.
\end{equation*}
Combine \eqref{eq:23111750} and \cref{28031805} (for $\lambda=\norm{u}_{\infty}^{-1}$) to obtain constants $C, \eta>0$ fulfilling 
\begin{equation} \label{eq:08021801} 
\forall s \in \R^d \, \forall n \in \N: \quad |\psi_{s+t_0} (g(t_n,t_0,s)u)|  \le C \, h_{\eta,\zeta}(t_n,s).
\end{equation}
At the same time, for any $s \in \Gamma_d \setminus \{\tau(s)=C_4(\zeta +\eta)\}$ and as $n \rightarrow \infty$, we verify that
\begin{equation*}
h_{\eta,\zeta}(t_n,s) \rightarrow 2 C \left(\tau(s)^{a \rho_1 - \gamma} + \tau(s)^{b \rho_1 - \gamma}  \right) \indikator{F(C_4(\zeta +\eta))}{s} =: \tilde{h}(s).
\end{equation*}
Finally, $\tilde{h}(\cdot)$ is integrable such that \eqref{eq:06021811} would follow due to a generalized version of the dominated convergence theorem (see Theorem 19 in Chapter 4 of \cite{RoyFi}), provided that $\int_{\R^d} h_{\eta,\zeta}(t_n,s) \, ds \rightarrow \int_{\R^d} \tilde{h}(s)\, ds$ which is obviously true. \\
For part (b) fix $t \in \Gamma_d$ and recall that the integrand in \eqref{eq:01021801} has been denoted by $f(t,\cdot)$. Then \eqref{eq:26011801} implies the following identity for any $s \notin \{0,t\}$. 
\begin{equation*}
\text{det}(f(t,s))=  \text{det} (\phi(-s)^{D(s)-qB(s)} ) \cdot \text{det} \left( \phi(\phi(-s)^{-E}(t-s))^{D(s)-qB(s)} -\mathbb{E}_m \right).
\end{equation*} 
Hence the assertion follows easily from \cref{11101701} (c) since $0 \notin \text{spec}(D(s)-qB(s))$ for any $s \in A$ and since 
\begin{equation*}
\text{spec}(r^{D(s)-qB(s)}- \mathbb{E}_m)=\{r^{\lambda}-1 : \lambda \in \text{spec } (D(s)-qB(s)) \}, \quad r>0.
\end{equation*}
\end{proof}
We now want to illustrate \cref{02031801}. Whereas its proof is in parts similar to the ones of Theorem 4.2 in \cite{paper} and Theorem 2.5 in \cite{lixiao}, the introduction of the operators $A(s)$ might be confusing at first. Actually, the following example emphasizes that neither $(C1)$ implies $(C2)$ nor the other way around. Overall our approach allows us to present a class of non-trivial examples in part (c) that admits a large choice of distribution families $(\mu_s)$.
\begin{example} \label{11041801}
\begin{itemize}
\item[(a)] Consider $m=2$ and $d=1$ together with $\phi(x)=|x|$. Hence, for $E=1$ and $q=1$, we see that $\phi$ is an $E$-homogeneous as well as $(1,E)$-admissible function. 
Moreover and similar to the setting in \cite{paper2} we assume $B(s)=B$ as well as $D(s)=D$ to be constant, respectively. For instance let $\eps=\sqrt{7}/12$, while 
\begin{equation*}
B= \frac{1}{2} \begin{pmatrix}  2 & 0 \\ 0 & 3 \end{pmatrix} \quad \quad \text{and} \quad \quad D= \frac{1}{4} \begin{pmatrix}  1 & 4 \eps \\ 4 \eps & 3  \end{pmatrix} .
\end{equation*}
Then we have $a=2/3$ and $b=1$ together with $-q/b=-1$ and $\beta-q/a=-1/2$. In addition it can be easily checked that
\begin{equation*}
\text{spec}(D-qB)=\{-3/4\pm \eps \} \subset (-1,-1/2)
\end{equation*}
which shows that $(C1)$ from \cref{02031801} is fulfilled. In contrast $(C2)$ fails, since
\begin{equation*}
BD= \frac{1}{8} \begin{pmatrix}   2 & 8 \eps \\  12 \eps &  9 \end{pmatrix}    \ne \frac{1}{8} \begin{pmatrix}   2 & 12 \eps \\ 8  \eps & 9 \end{pmatrix}  = DB.
\end{equation*}
However, using Weyl's inequality (see \cite{horn} for example) and \eqref{eq:12071805}, it follows that \eqref{eq:22031801} always implies \eqref{eq:22031802} and that we necessarily have $D(s) \in Q(\R^m)$ in the context of \cref{02031801}.
\item[(b)] Define $E$ and $\phi$ as in part (a), also let $m=2$ and $B=B(s)=\text{diag}(1, 3/2)$ again. Then, for $D=D(s)=\text{diag}(1/4,1/4)$, we see that condition $(C2)$ of \cref{02031801} is fulfilled. Though we still have $-q/b=-1$ and $\beta-q/a=-1/2$. Hence $(C1)$ does not hold true, since $D-qB=\text{diag}(-3/4,-5/4)$ which leads to $\lambda_{D-qB}=-5/4<-q/b$.
\item[(c)] For any choice of distributions $(\mu_s)$ with symmetric exponents $B(s)$ that fulfill the initial assumptions from \cref{chapter2} we consider the operator $E=\text{diag}(e_1,...,e_d)$, where $e_1,...,e_d \ge 1$. Moreover, let $D(s)=\Delta (s) B(s)$ for a function $\Delta:\R^d \rightarrow [\delta_0,\delta_1] \subset (0, a)$. Then we may use $\phi:\R^d \rightarrow [0,\infty)$, defined by $(x_1,...,x_d) \mapsto \summezwei{j=1}{d} |x_j|^{1/e_j}$. It follows in view of Corollary 2.12 in \cite{bms} that $\phi$ is $E$-homogeneous as well as $(1,E)$-admissible. Particularly, $(C2)$ and the remaining assumptions of \cref{02031801} are fulfilled. Finally note that $d \ge 2 $ is sufficient for $q  > \delta_1$ in this case which means that the resulting moving-average representation is full due to \cref{09021801}. 
\end{itemize}
\end{example}
\begin{remark} \label{14031811}
Most of the moving-average representations known in literature demand $D=D(s)$ to be constant. Then $D$ is called \textit{space-scaling exponent} or \textit{Hurst-Index}. At the same time the underlying random measures are often stable ones such that we have $B(s)=\alpha^{-1}\mathbb{E}_m$ for some $0< \alpha<2$ (see \cite{lixiao}). In this case $(C2)$ reduces to \eqref{eq:22031802} and equals $0<\lambda_D \le \Lambda_D < \beta$. Hence \cref{02031801} covers the univariate and multivariate $\alpha$-stable (non-Gaussian) moving-average representations as they were proposed by \cite{bms}, \cite{lixiao} and Example 3.6.5 in \cite{SaTaq94}.  
\end{remark}
Now in our context $B(s)$ generally denotes an exponent of a symmetric operator-stable distribution and does not even have to be diagonal. This means that the possibility for $D(s)$ to vary in \textit{time} enlarges the possible choice of exponents $B(s)$ such that $D(s)B(s)=B(s)D(s)$ holds and such that we are not restricted to \eqref{eq:22031801}. However, if $B(s)=B$ and $D(s)=D$ are constant, then \eqref{eq:22031801} reduces to
\begin{equation*}
0<\lambda_{D-qB}+\lambda_{qB} \le  \Lambda_{D-qB}+\Lambda_{qB}< \beta.
\end{equation*}
On one hand this is exactly the condition that has been required in Theorem 4.2 of \cite{paper2} and the corresponding operator-stable moving-average representation equals \eqref{eq:01021801} accordingly (at least in the symmetric case). On the other hand the consideration of $(C_2)$ and its possible advantages have been missed in \cite{paper2}. \\ \quad \\
A similar omission can be observed in the context of Proposition 4.3 in \cite{falconer3}, where a class of univariate multi-stable processes has been investigated that is closely related to the several moving-average representations we mentioned above. For the rest of this section we want to find a suitable extension to our multi operator-stable point of view. \\
Throughout let $d=1$. Then we consider the $1$-homogeneous functions $x \mapsto (x)_{+}=\max \{x,0 \}$ and $x \mapsto (x)_{-}=\max \{-x,0\}$. Moreover, for $t,b^{+}, b^{-} \in \R$ arbitrary and $D(s),B(s)$ as before, this allows to use integrands of the form
\begin{equation*}
f(t,s):=b^{+} ((t-s)_{+}^{D(s)-B(s)} - (-s)_{+}^{D(s)-B(s)}) + b^{-} ((t-s)_{-}^{D(s)-B(s)} - (-s)_{-}^{D(s)-B(s)}),
\end{equation*}
provided that they are integrable with respect to $\mathbb{M}$. Here, for $E \in$ L$(\R^m)$, we generally define $0^E:=0 \in \Li{m}$ (even if $\lambda_E<0$). In the \textit{well-balanced} case $b{+}=b^{-}$ this essentially leads to \cref{02031801} with $\phi(x)=|x|$. Unfortunately, the functions $\phi(x)=(x)_{+}$ and $\phi(x)=(x)_{-}$ are not admissible, since $\phi(x)=0$ does not imply $x=0$. Thus \cref{02031801} fails for $b^{+}\ne b^{-}$. However, the following result can be obtained, where it is convenient to restrict ourselves to the case $b^{+}=1$ and $b^{-}=0$. Note that $(C1')$ and $(C2')$ are nothing else than $(C1)$ and $(C2)$ for $q=\beta=1$, respectively.
\begin{cor} \label{28051801}
Let $D(s) \in \Li{m}$ be symmetric for every $s \in \R$ and assume that either condition $(C1')$ or $(C2')$ is fulfilled, namely:
\begin{itemize}
\item[\textbf{(C1$'$)}] We have that
\begin{equation*} 
-1 / b< \inf \nolimits_{s \in \R} \lambda_{D(s)-B(s)} \le  \sup \nolimits_{s \in \R} \Lambda_{D(s)-B(s)}<1 -1 / a.
\end{equation*}
\item [\textbf{(C2$'$)}] For almost every $s \in \R$ we have that $B(s)D(s)=D(s)B(s)$ together with
\begin{equation} \label{eq:10081802} 
0< \inf \nolimits_{s \in \R} \lambda_{D(s)} \le \sup \nolimits_{s \in \R} \Lambda_{D(s)}< 1.
\end{equation}
\end{itemize}
Then the stochastic integral 
\begin{equation}  \label{eq:10081801}
X(t):=\integral{-\infty}{\infty}{ \left[ (t-s)_{+}^{D(s)-B(s)}-(-s)_{+}^{D(s)-B(s)} \right]}{\mathbb{M}(ds)}
\end{equation}
exists for every $t \in \R$ and the resulting $\R^m$-valued process $\mathbb{X}=\{X(t):t \in \R\}$ is stochastically continuous. Furthermore, $\mathbb{X}$ satisfies the assumptions of \cref{27041805} and therefore has a continuous version, if
\begin{equation} \label{eq:24051801}
b/a^2-1/a< \inf \nolimits_{s \in \R} \lambda_{D(s)-B(s)} \le  \sup \nolimits_{s \in \R} \Lambda_{D(s)-B(s)}<1 -1 / a
\end{equation}
is fulfilled instead of $(C1')$ or $(C2')$, respectively. Certainly, \eqref{eq:24051801} implies that $a>1$.
\end{cor}
\begin{proof}
Although $\phi$ is not admissible, the existence of $\mathbb{X}$ can be checked similar to the proof of \cref{02031801}. Particularly, if we assume $(C1')$, this can mostly be attained by arguing as below for $t_1=t$ and $t_2=0$. The details are left to the reader. \\
In order to prove the additional statement we fix $K>0$ and $t_1,t_2 \in [-K,K]$. Without loss of generality we may assume that $t_1>t_2$, i.e. $|t_1-t_2|=t_1-t_2 \le 2K$. Let $A(s):=D(s)-B(s)$ for any $s \in \R$ and replace the notation from above by
\begin{equation*}
\rho_1:=\inf \nolimits_{s \in \R} \lambda_{A(s)}, \quad \rho_2:=\sup \nolimits_{s \in \R} \Lambda_{A(s)},
\end{equation*}
where $0< \rho_1 \le \rho_2 <1$ due to \eqref{eq:24051801}. Also define 
\allowdisplaybreaks
\begin{align*}
f(s): &=(t_1-s)_{+}^{D(s)-B(s)} - (t_2-s)_{+}^{D(s)-B(s)} \\
&= \begin{cases} (t_1-s)^{A(s)} - (t_2-s)^{A(s)}, & \text{$s<t_2$} \\ (t_1-s)^{A(s)} , & \text{$t_2 \le s \le t_1$} \\ 0, & \text{$s>t_1$.} \end{cases} 
\end{align*}
Observe that $\xi:=b^{-1}(a \rho_1 +1)>a^{-1}$ according to \eqref{eq:24051801} again. Hence by linearity and in view of \eqref{eq:12021801} the assertion would particularly follow, if 
\begin{equation} \label{eq:07081801}
\integral{\R}{}{  \norm{f(s)}^{a,b} }{ds} \le C \, (t_1-t_2)^{a \rho_1+1}
\end{equation}
holds true. Note that $C>0$ can be chosen independent of $t_1$ and $t_2$, since the same applies to the constants $C_1,C_2,...$ that occur in sequel. By a change of variables and using \cref{24111750} we first verify that
\begin{equation} \label{eq:07081811}
\integral{t_2}{t_1}{ \norm{(t_1-s)^{A(s)}}^{a,b} }{ds}  \le C_1 \integral{0}{t_1-t_2}{ \left(s^{a \lambda_{A(s+t_1)}}+s^{b \lambda_{A(s+t_1)}}\right) }{ds} \le  C_2  (t_1-t_2)^{a \rho_1+1},
\end{equation}
while we can write
\begin{align*}
\integral{-\infty}{t_2}{ \norm{(t_1-s)^{A(s)} - (t_2-s)^{A(s)}}^{a,b} }{ds} & = \integral{0}{\infty}{ \norm{(s+t_1-t_2)^{A(t_2-s)} - s^{A(t_2-s)}}^{a,b} }{ds}.
\end{align*}
Likewise and in view of $\norm{A(\cdot)} \le \rho_2$ we compute for any $s \ge t_1-t_2$ that
\allowdisplaybreaks
\begin{align*}
\norm{(s+t_1-t_2)^{A(t_2-s)} - s^{A(t_2-s)}} & \le \integral{s}{s+t_1-t_2}{ \norm{A(t_2-s) \, y^{A(t_2-s)-\mathbb{E}_m}} }{dy} \\
& \le C_3 \integral{s}{s+t_1-t_2}{ (y^{\rho_1-1}+y^{\rho_2-1} ) }{dy} \\
& \le C_4 (t_1-t_2) (s^{\rho_1-1} + s^{\rho_2-1}).     
\end{align*}
Hence we get that
\allowdisplaybreaks
\begin{align}
\integral{t_1-t_2}{\infty}{ \norm{(s+t_1-t_2)^{A(t_2-s)} - s^{A(t_2-s)}}^b }{ds} & \le C_5 (t_1-t_2)^b ((t_1-t_2)^{b(\rho_1-1)+1}+(t_1-t_2)^{b(\rho_2-1)+1}) \notag \\
& \le C_6  (t_1-t_2)^{b\rho_1+1} \notag \\
& \le C_7  (t_1-t_2)^{a\rho_1+1} \label{eq:07081812},
\end{align}
since $b(\rho_1-1) \le b(\rho_2-1) <-1$ due to \eqref{eq:24051801}. Moreover, it is even easier to verify that 
\begin{equation}
\integral{t_1-t_2}{\infty}{ \norm{(s+t_1-t_2)^{A(t_2-s)} - s^{A(t_2-s)}}^a }{ds}  \le C_8  (t_1-t_2)^{a\rho_1+1}. \label{eq:07081813}
\end{equation}
Finally, if $0 \le s \le t_1-t_2$, we use that $\R_{+} \ni y \mapsto y^{\rho_1}$ is sub-additive. Similar as before this implies that
\allowdisplaybreaks
\begin{align*}
\norm{(s+t_1-t_2)^{A(t_2-s)} - s^{A(t_2-s)}} & \le C_9 \integral{s}{s+t_1-t_2}{y^{\rho_1-1}}{dy} \\
& = C_{10} ((s+t_1-t_2)^{\rho_1} - s^{\rho_1}) \\
& \le C_{10} (t_1-t_2)^{\rho_1}
\end{align*}
and therefore that
\begin{equation}  \label{eq:07081814}
\integral{0}{t_1-t_2}{ \norm{(s+t_1-t_2)^{A(t_2-s)} - s^{A(t_2-s)}}^{a,b} }{ds} \le C_{11}  (t_1-t_2)^{a\rho_1+1}.
\end{equation}
Overall \eqref{eq:07081801} follows from \eqref{eq:07081811}-\eqref{eq:07081814}.
\end{proof}
\begin{remark}
If $m=1$ and $D=D(s)$ is constant, \eqref{eq:10081801} leads to the so called \textit{linear fractional multi-stable motion} (LFMSM) from \cite{falconer3} which we already announced above. We also mentioned that Proposition 4.3 in \cite{falconer3} neglects the possibility of $(C2')$. Nevertheless, it is easy to see that $(C1')$ becomes $1/a-1/b<D<1+1/b-1/a$ in this case, which is exactly the condition that has been required in \cite{falconer3} for the LFMSM in order to exist. Furthermore, a sufficient condition for the LFMSM to have a continuous version is suggested there, namely
\begin{equation} \label{eq:10081803}
1/a<D<1+1/b- 1/a.
\end{equation}
In contrast we compute that \eqref{eq:24051801} equals $b/a^2<D<1+1/b- 1/a$ in this context which is more restrictive than \eqref{eq:10081803}. We think that this gap may arise from a possible mistake in Proposition 3.1 and especially (3.3) of \cite{falconer3}. 
\end{remark}
%% Section 4: 
\section{Tangent fields} \label{chapter4}
In this section we want to explore the local form of certain random fields which is inspired by the univariate considerations in \cite{falconer3}. There two types of \textit{localisability} have been proposed, where the stronger one assumes that the underlying stochastic processes have continuous versions. However, if $d>1$ or $a \le 1$, the criteria from \cref{27041805} will fail for the random fields that we have constructed in \cref{chapter3}. And in view of Remark 4.5 in \cite{paper2} it is even possible that there may not exist any continuous version in general. Therefore we will give a definition that neglects this aspect in sequel. Throughout let $|\cdot|$ be the Lebesgue measure on $(S,\Sigma)=(\R^d, \B(\R^d))$ again, where $\mathbb{M}$ is an ISRM as before.
\begin{defi}
Fix $u \in \R^d$ and let $\mathbb{X}=\{X(t):t \in \R^d\}$ as well as $\mathbb{X}'_u=\{X'_u(t):t \in \R^d\}$ be $\R^m$-valued random fields. Then, for $E \in Q(\R^d)$ and $D \in Q(\R^m)$, we say that $\mathbb{X}$ is \textit{$(E,D)$-localisable at u} with \textit{local form}/\textit{tangent field} $\mathbb{X}'_u$, if the following convergence holds.
\begin{equation} \label{eq:14081801}
\{r^{-D}(X(u+r^{E}t)-X(u)): t \in \R^d \} \overset{fdd}{\longrightarrow} \{X_u'(t):t \in \R^d\} \quad (r \rightarrow 0).
\end{equation}
Here $\overset{fdd}{\rightarrow}$ means convergence of all finite-dimensional distributions. If, for almost every $u \in \R^d$, $\mathbb{X}$ is localisable at $u$ with some operator-stable tangent field $\mathbb{X}'_u$ (see Definition 4.1 in \cite{paper2}), we call $\mathbb{X}$ \textit{multi operator-stable}. 
\end{defi}
Obviously, the operators $E$ and $D$ are not unique. Hence the local forms may vary, too. However, for $d=1$, we may always assume that $E=1$ and in this case we get back the corresponding definition from \cite{falconer3}. On the other hand we have two notions of \textit{multi operator-stability}, depending on if we are dealing with ISRMs (see \cref{04101702}) or with random fields. \cref{13081801} will show that both concepts can be matched in some sense. \\ \quad \\
Now consider a univariate stochastic process which is defined by a stochastic integral with respect to a multi-stable random measure. Then Proposition 3.2 in \cite{falconer3} states sufficient conditions for localisability and also describes possible local forms. Actually, these conditions would look quite complicated in our multivariate setting due to the fact that linear operators do not commutate in general. Instead we want to investigate particular examples that arise from \cref{02031801}. For this purpose observe that \eqref{eq:14081801} implies that $\mathbb{X}'_u$ is (strictly) $(E,D)$-operator-self-similar in the sense of \cite{paper2}, i.e. for any $c>0$ we have that 
\begin{equation*}
\{X_u'(c^E t):t \in \R^d\} \overset{fdd}{=} \{c^D X_u'( t):t \in \R^d\},
\end{equation*}
where $\overset{fdd}{=}$ means equality of all finite-dimensional distributions (see \eqref{eq:25091801}). Since the random fields from Theorem 4.2 in \cite{paper} are operator-stable and at the same time operator-self-similar, these could be appropriate candidates to serve as local forms for \eqref{eq:01021801}. Hence we obtain the following statement. Note that similar results can be derived for the random fields from \cref{28051801}, the details are left to the reader.
\begin{theorem} \label{19021801}
Fix $u \in \R^d$. Assume that the assumptions of \cref{02031801} are fulfilled with condition (C2) and denote the random field from \eqref{eq:01021801} by $\mathbb{X}=\{X(t):t \in \R^d \}$. Moreover, suppose that the mappings $s \mapsto B(s)$ and $s \mapsto D(s)$ are continuous at $u$ such that the following conditions are fulfilled for $v(r,s):=r^{D(u+r^Es)} r^{-D(u)}$:
\begin{itemize}
\item[(i)] $\limes{r}{0} \,v (r,s)=\mathbb{E}_m$ for every $s \in \R^d$.
\item[(ii)] There exists a constant $C>0$ such that $\norm{v (r,s)} \le C$ holds true for every $s \in \R^d$ and $r>0$ (small).
\end{itemize}
Also recall $\mathbb{M}_u$ from \cref{09011805}. Then $\mathbb{X}$ is $(E,D(u))$-localisable at $u$ with local form $\mathbb{X}'_u=\{X'_u(t):t \in \R^d\}$, where 
\begin{equation} \label{eq:18081801}
X'_u(t)=\integral{\R^d}{}{\left[ \phi(t-s)^{D(u)-qB(u)}-\phi(-s)^{D(u)-qB(u)} \right] }{\mathbb{M}_u(ds)}.
\end{equation}
Hence $\mathbb{X}'_u$ equals the operator-stable moving-average representation from Theorem 4.2 in \cite{paper2}.
\end{theorem}
\begin{proof}
For $n \in \N$ arbitrary fix $t_1,...,t_n \in \R^d$ and $\theta_1,...,\theta_n \in \R^m$. Recall the definition of $g(\cdot,\cdot,\cdot)$ from \eqref{eq:16081805} and that $B(s)$ as well as $D(s)$ are symmetric for any $s \in \R^d$. Then by linearity, \eqref{eq:17111740} and a change of variables the log-characteristic function of the random vector
\begin{equation*}
(r^{-D(u)}(X(u+r^E t_1)-X(u)),...,r^{-D(u)}(X(u+r^E t_n)-X(u)))
\end{equation*}
can be computed as follows for every $r>0$. Throughout let $t_0=t_0(r,s):=u+r^Es-s$.
\allowdisplaybreaks
\begin{align}
& \integral{\R^d}{}{\psi_s \left( \summezwei{j=1}{n} \left (  \phi(u+r^{E}t_j-s )^{D(s)-qB(s)} - \phi(u-s )^{D(s)-qB(s)} \right ) r^{-D(u)} \theta_j  \right)  }{ds} \notag \\
& = \int \nolimits_{\R^d}^{} r^q \cdot \psi_{u+r^E s} \left( \summezwei{j=1}{n} \left (  \phi(r^{E} (t_j-s) )^{D(u+r^{E} s)-qB(u+r^{E} s)} \right. \right. \notag \\
& \left. \left. \qquad \qquad \qquad \qquad \qquad - \phi(-r^E s )^{D(u+r^{E} s)-qB(u+r^{E} s)} \right) r^{-D(u)} \theta_j  \right)  \, ds \notag  \\
& = \int \nolimits_{\R^d}^{} \psi_{u+r^E s} \left( r^{q B(u+r^E s)}   \summezwei{j=1}{n}  \left(\phi(t_j-s )^{D(u+r^{E} s)-qB(u+r^{E} s)} \right. \right. \notag \\
& \left.  \left.  \qquad \qquad \qquad \qquad \qquad - \phi(- s )^{D(u+r^{E} s)-qB(u+r^{E} s)} \right) r^{D(u+r^{E} s)-qB(u+r^{E} s)} r^{-D(u)} \theta_j  \right)  \, ds \notag  \\
& = \int \nolimits_{\R^d}^{} \psi_{s+t_0} \left(  \summezwei{j=1}{n} g(t_j,t_0,s)v (r,s) \theta_j  \right)  \notag .
\end{align} 
Here we also used that $B(s)D(s)=D(s)B(s)$ (almost everywhere) and the fact that $\phi$ as well as $\psi_s$ are certain homogeneous functions. Then, in view of L\'{e}vy's continuity theorem and \eqref{eq:17111740} again, it suffices to show that the following convergence holds true as $r \rightarrow 0$.
\begin{align} 
&\integral{\R^d}{}{ \psi_{s+t_0} \left(  \summezwei{j=1}{n} g(t_j,t_0,s)v (r,s) \theta_j  \right) }{ds} \notag \\
& \qquad \rightarrow \int \nolimits_{\R^d}^{} \psi_u \left(  \summezwei{j=1}{n} \left ( \phi(t_j-s )^{D(u)-qB(u)}  - \phi(- s )^{D(u)-qB(u)} \right )  \theta_j  \right)  \, ds.  \label{eq:20021805} 
\end{align}
By assumption we have for every $s \in \R^d$ that $B(u+r^Es) \rightarrow B(u)$. Similar to the proof of \cref{09011805} this implies that $\mu_{u + r^E s}$ converges to $\mu_u$ weakly and therefore that $\psi_{u+r^E s} (\cdot) =\psi_{s+t_0} (\cdot)\rightarrow \psi_u(\cdot)$ holds uniformly on compact sets as $r \rightarrow 0$ (see Lemma 3.1.10 in \cite{thebook}). At the same time we have for every $s \in \R^d$ that $D(u+r^Es) \rightarrow D(u)$ as well as $v (r,s) \rightarrow \mathbb{E}_m$. Hence, as $r \rightarrow 0$, we observe for almost every $s \in \R^d$ that 
\begin{align*}
 \psi_{s+t_0} \left(  \summezwei{j=1}{n} g(t_j,t_0,s) v(r,s) \theta_j  \right) &  \rightarrow \psi_u \left(  \summezwei{j=1}{n} \left ( \phi(t_j-s )^{D(u)-qB(u)}  - \phi(- s )^{D(u)-qB(u)} \right )  \theta_j  \right).
\end{align*}
Now in view of condition (ii) there exists some $\lambda>0$ with $\norm{v  (r,s) \theta_j }_{\infty} \le \lambda^{-1}$ for any $s \in \R^d, r>0$ (small) and $j=1,...,n$. Hence we can use \eqref{eq:23111750} and \eqref{eq:16081801} to verify that the following estimates are valid for any $s \in \R^d$ and $r>0$ (small), where $C>0$ is sufficiently large.
\allowdisplaybreaks
\begin{align}
 \left|   \psi_{s+t_0} \left(  \summezwei{j=1}{n} g(t_j,t_0,s) v(r,s) \theta_j  \right) \right|   &   \le  C \summe{j=1}{n}  \tau_{s+t_0} (g(t_j,t_0,s) v (r,s) \theta_j)  \label{eq:17081801} \\
 &   \le  C \summe{j=1}{n}  \underset{\norm{\theta}_{\infty} \le \lambda^{-1} }{\sup} \tau_{s+t_0} (g(t_j,t_0,s)  \theta). \notag
 \end{align}
Now use \cref{28031805} to see that \eqref{eq:20021805} follows by dominated convergence.
\end{proof}
In general, the continuity assumptions on $B(s)$ and $D(s)$ should appear natural in the context of \cref{19021801}, while (i) and particularly (ii) are rather challenging conditions. Actually, this problem vanishes, if $D(s)$ is constant. On the other hand this restricts the possible choices for $B(s)$, since we essentially benefited from $B(s)D(s)=D(s)B(s)$ which is part of condition $(C2)$.
\begin{example} \label{16081810}
Fix $u \in \R^d$. Throughout consider $E \in Q(\R^d)$ with $q=\text{tr}(E)$ and let $\phi$ be an $E$-homogeneous as well as $(\beta,E)$-admissible function for some $\beta>0$. Moreover, assume that the mapping $s \mapsto B(s)$ is continuous at u. We now give examples that fit into \cref{19021801} (and therefore also into \cref{02031801}).
\begin{itemize}
\item[(a)] As mentioned in \cref{14031811} it is common to assume that $D=D(s)$ is constant. Moreover, if $D$ has the particular form $D=H \,  \mathbb{E}_m$, we merely need that $0< H < \beta$. 
\item[(b)] If $\mathbb{M}$ is multivariate multi-stable (see \cref{11071802} (b)) with $D=D(s)$ still being constant, the symmetric operator $D \in Q(\R^m)$ can be chosen arbitrary as long as
\begin{equation} \label{eq:29081801}
0< \lambda_D \le  \Lambda_D< \beta
\end{equation}
holds true. If $B(s)$ is only diagonal (almost everywhere) instead, it accordingly suffices to assume that $D \in Q(\R^m)$ is diagonal, provided that \eqref{eq:29081801} is still fulfilled.
\item[(c)] Following \cref{11041801}, where $D(s)$ is not constant, we may also use the operators $D(s)=\Delta(s) B(s)$ for some function $\Delta:\R^d \rightarrow [\delta_0,\delta_1] \subset (0, a \beta)$ which is at least continuous at $u$. However, the conditions (i) and (ii) are hard to check in this case. Particularly, $\Sigma(r,s)$ can not be simplified in general.
\item[(d)] Finally let us generalize part (a) by considering $D(s)=\delta(s) \mathbb{E}_m$ for some continuous function (continuous at $u$, respectively) $\delta:\R^d \rightarrow [\delta_0,\delta_1] \subset (0, \beta)$. Then, similar to \eqref{eq:25071802}, a sufficient condition for (i) of \cref{19021801} is given by 
\begin{equation} \label{eq:29081820} 
|\delta(u+s) - \delta(u)| = o (1/ \ln \norm{s}) \quad (s \rightarrow 0).
\end{equation}
At the same time \eqref{eq:29081820} ensures that $v(r,s)$ is bounded for any $r>0$ (small) and $s \in K$, where $K \subset \R^d$ is compact. Fortunately, for large $s$, the required boundedness of $v(r,s)$ can be relaxed. For this purpose let us recall \cref{28031805} and especially the functions $h_{\eta,\zeta}(t,\cdot)$. Then $(C2)$ implies that we have $\gamma=q$ together with
\begin{equation*}
\rho_2=\sup_{s \in \R^d} \Lambda_{D(s)}=\sup_{s \in \R^d} \delta(s)<\beta
\end{equation*}
in this context. Moreover, for any $0<\eps<\beta- \rho_2$, we observe that the functions $s \mapsto \tau(s)^{a \eps}h_{\eta,\zeta}(t,s) $ are still integrable at infinity. Hence, based on \eqref{eq:17081801} and in view of \eqref{eq:27111701}, we merely need the existence of a constant $C>0$ and some $s_0>0$ such that the following inequality holds for any $r>0$ (small) and $\norm{s}\ge s_0$:
\begin{equation} \label{eq:29081810}
\norm{v(r,s)} \le C \tau(s)^{a \eps /b}.
\end{equation}
In our opinion, reformulating \eqref{eq:29081810} by feasible assumptions in terms of the function $\delta$, turns out to be an open and interesting problem.
\end{itemize}
\end{example}
% SECRET SECRET SECRET
\secret{
\begin{remark}
Recall \cref{28031805} and especially the functions $h_{\eta,\zeta}(t,\cdot)$ for $\eta,\zeta>0$ and $t \in \R^d$. Then $(C2)$ implies that $\gamma=q$ together with
\begin{equation*}
0<\rho_1= \inf_{s \in \R^d} \lambda_{D(s)} \le \sup_{s \in \R^d} \Lambda_{D(s)}=\rho_2<\beta
\end{equation*}
 in this context. Hence, for $0<\eps_1<\rho_1$ and $0<\eps_2<\beta- \rho_2$, the functions $s \mapsto \tau(s)^{-a \eps_1}h_{\eta,\zeta}(t,s) $ are still integrable near zero, while the same applies to the functions $s \mapsto \tau(s)^{a \eps_2}h_{\eta,\zeta}(t,s) $ at infinity. Based on \eqref{eq:17081801} and in view of \eqref{eq:27111701} this allows to relax the boundedness of $r^{D(u+r^Es)} r^{-D(u)}$. More precisely, we merely need the existence of constants $C_1,C_2>0$ such that the following inequalities hold for any $r>0$ (small). Namely
\begin{equation*}
\norm{r^{D(u+r^E s)} r^{-D(u)}} \le C_1 \tau(s)^{- a \eps_1 /b},
\end{equation*}
whenever $s$ is small and
\begin{equation*}
\norm{r^{D(u+r^E s)} r^{-D(u)}} \le C_2 \tau(s)^{a \eps_2 /b},
\end{equation*}
whenever $s$ is large.
\end{remark}
}
Actually, \cref{02031801} allows that $D(s)=B(s)$. Then, if $d=\beta=E=1$ for example, this may cause that the random fields defined in \eqref{eq:01021801} and especially in \eqref{eq:18081801} are not full. On the other hand the tangent fields from \cref{19021801} are $(1,D(u))$-operator-self-similar in this case. We want to finish with the following result which generalizes Example 4.1 in \cite{falconer3} and which leads to tangent fields that are $(1,B(u))$-operator-self-similar instead. Also note that similar observations as in \cref{11071802} and \cref{16081810} are valid. Particularly, \eqref{eq:11071801} boils down to (4.1) in \cite{falconer3}, if $B(s)= \alpha(s)^{-1} \mathbb{E}_m$ for every $s \in \R$.
\begin{prop} \label{13081801}
Fix $u \in \R$ and let $\mathbb{M}$ be as before, where $d=1$. Assume that $\R \ni s \mapsto B(s)$ is continuous at $u$ such that
\begin{equation}  \label{eq:11071801}
\limes{r}{0} r^{B(u+rs)} r^{-B(u)} =\mathbb{E}_m  
\end{equation} 
holds uniformly in $s$ on compact subsets of $\R$. Let $\indikatorzwei{[s,t]}:=-\indikatorzwei{[t,s]} $ for $s>t$. Moreover, let $w:\R \rightarrow \Li{m}$ be a continuous function with $w(s)B(s)=B(s)w(s)$ almost everywhere and define the random field $\mathbb{Y}=\{Y(t):t \in \R\}$ by
\begin{equation*}
Y(t)=\integral{\R}{}{\indikator{[0,t]}{s} w(s)  }{\mathbb{M}(ds)}.
\end{equation*}
Then $\mathbb{Y}$ is $(1,B(u))$-localisable at $u$ with local form $\mathbb{Y}'_u=\{Y'_u(t):t \in \R\}$ given by
\begin{equation*}
Y_u'(t)=\integral{\R}{}{ \indikator{[0,t]}{s} w(u)  }{\mathbb{M}_u(ds)}=w(u) \mathbb{M}_u([0,t]).
\end{equation*}
Here $\mathbb{M}_u$ is defined as before such that $\{\mathbb{M}_u([0,t]):t \in \R\}$ is an $\R^m$-valued operator-stable L\'{e}vy process with exponent $B(u)$. Moreover, $w(u) \in$ GL$(\R^m)$ implies that $\mathbb{Y}'_u$ is full due to \cref{11101701}.
\end{prop}
\begin{proof}
Fix $t_1,...,t_n \in \R$ as well as $\theta_1,...,\theta_n \in \R^m$ again. Using the convention above it follows that $\indikatorzwei{[s,t]} -\indikatorzwei{[s,u]}=\indikatorzwei{(u,t]} $ for any $s,t,u \in \R$. Hence, similar to the proof of \cref{19021801}, it suffices to show that
\begin{equation} \label{eq:11071803}
 \integral{\R}{}{\psi_s \left( \summezwei{j=1}{n} \indikator{(u,u+r t_j]}{s} w(s)^{*} r^{-B(u)} \theta_j  \right)  }{ds} \rightarrow  \integral{\R}{}{\psi_u \left( \summezwei{j=1}{n} \indikator{[0,t_j]}{s} w(u)^{*} \theta_j  \right)  }{ds}
\end{equation}
holds as $r \rightarrow 0$. Using $w(s)B(s)=B(s)w(s)$ and the $B(s)$-homogeneity of $\psi_s$ we see by a change of variables for any $r>0$ that
\allowdisplaybreaks
\begin{align*} 
& \integral{\R}{}{\psi_s \left( \summezwei{j=1}{n} \indikator{(u,u+r t_j]}{s} w(s)^{*} r^{-B(u)} \theta_j  \right)  }{ds} \\
&=  \integral{\R}{}{r \cdot \psi_{u+rs} \left( \summezwei{j=1}{n} \indikator{(u,u+r t_j]}{u+rs} w(u+rs)^{*} r^{-B(u)} \theta_j  \right)  }{ds}  \\
&=  \integral{\R}{}{ \psi_{u+rs} \left( \summezwei{j=1}{n} \indikator{(0, t_j]}{s} w(u+rs)^{*} r^{B(u+rs)} r^{-B(u)} \theta_j  \right)  }{ds} .
 \end{align*}
Moreover, by \eqref{eq:11071801} and the stated continuity assumptions, we can argue as in the proof of \cref{19021801} that the following convergence is valid for almost every $s \in \R$ as $r \rightarrow 0$.
\begin{equation} \label{eq:12071801}
\psi_{u+rs} \left( \summezwei{j=1}{n} \indikator{(0, t_j]}{s} w(u+rs)^{*} r^{B(u+rs)} r^{-B(u)} \theta_j  \right)  \rightarrow \psi_{u} \left( \summezwei{j=1}{n} \indikator{[0, t_j]}{s} w(u)^{*}  \theta_j  \right) .
\end{equation}
Finally, there is a compact set $K \subset \R$ (independent of $r$) such that the left-hand side of \eqref{eq:12071801} vanishes for $s \notin K$, while $w(u+rs)$ and $r^{B(u+rs)} r^{-B(u)}$ are bounded for any $s \in K$ and $r>0$ (small) by assumption. Hence the left-hand side of \eqref{eq:12071801} is also bounded (combine \eqref{eq:24111732} and \eqref{eq:23111750} for example) such that dominated convergence implies \eqref{eq:11071803}.
\end{proof}
\begin{remark}
Let $w(s)=\mathbb{E}_m$ for all $s \in \R$. Then $\mathbb{Y}=\{\mathbb{M}([0,t]):t \in \R\}$ is what is called an additive process (in law), that is it has independent but not necessarily stationary increments (see Definition 1.6 in \cite{Sato}). Then \cref{13081801} states that in this particular case $\mathbb{Y}$ is $(1,B(u))$-localisable at $u$ with local form $\mathbb{Y}'_u=\{\mathbb{M}_u([0,t]):t \in \R\}$ being an operator-stable L\'{e}vy process with exponent $B(u)$.
\end{remark} 
%% References

\end{document}